\documentclass[10pt]{amsart}
\usepackage{cm-math}
\usepackage{cm-style}

\usepackage[a4paper]{geometry}

\makeatletter
\renewcommand\section{\@startsection{section}{1}%
  \z@{.7\linespacing\@plus\linespacing}{.5\linespacing}%
  {\large\rmfamily\scshape\centering}}
\makeatother

\allowdisplaybreaks[4]

\begin{document}
\title[ML Methods for UQ of Elliptic PDEs with Random Anisotropic Diffusion]{Multilevel Methods for Uncertainty Quantification of Elliptic PDEs with Random Anisotropic Diffusion}

\author{Helmut Harbrecht \and Marc Schmidlin}
\address{
Helmut Harbrecht and Marc Schmidlin,
Departement Mathematik und Informatik,
Universit\"at Basel,
Spiegelgasse 1, 4051 Basel, Switzerland
}
\thanks{\textbf{Funding:} The work of the authors was supported
  by the Swiss National Science Foundation (SNSF)
  through the project
  ``Multilevel Methods and Uncertainty Quantification in Cardiac Electrophysiology''
  (grant \mbox{205321\_169599}).}
\email{\{helmut.harbrecht, marc.schmidlin\}@unibas.ch}

%
%
\begin{abstract}
  We consider elliptic diffusion problems with a random anisotropic diffusion 
  coefficient, where, in a notable direction given by a random vector field,
  the diffusion strength differs from the diffusion strength perpendicular
  to this notable direction. The Karhunen\hyp{}Lo\`eve expansion then 
  yields a parametrisation of the random vector field and, therefore, also 
  of the solution of the elliptic diffusion problem. We show that, given 
  regularity of the elliptic diffusion problem, the decay of the 
  Karhunen\hyp{}Lo\`eve expansion entirely determines the regularity 
  of the solution's dependence on the random parameter, also when 
  considering this higher spatial regularity. This result then implies that
  multilevel collocation and multilevel quadrature methods may be used 
  to lessen the computation complexity when approximating quantities of 
  interest, like the solution's mean or its second moment, while still yielding 
  the expected rates of convergence. Numerical examples in three spatial 
  dimensions are provided to validate the presented theory.
\end{abstract}


\maketitle
%
\section{Introduction}
The numerical approximation of quantities of interest, such 
as expectation, variance, or more general output functionals, 
of the solution of a diffusion problem with a scalar random diffusion 
coefficient with multilevel collocation or multilevel quadrature methods 
has been considered previously, see e.g.\ \cite{BSZ11,CGST11,GKNSSS15,%
GHP15,HPS13,HPS16b,KSS15,TJWG} and the references 
therein; in this \emph{isotropic} case, the mixed smoothness required for 
the use of such multilevel methods has been provided in \cite{CDS10}
for uniformly elliptic diffusion coefficients and in \cite{HS14} for 
log\hyp{}normally distributed diffusion coefficients.
However, in simulations of certain diffusion phenomena in science 
and engineering, the diffusion that needs to be modeled may not 
necessarily be isotropic. One specific application we have in mind 
here stems from cardiac electrophysiology, where the electrical 
activation of the human heart is considered. It is known that the 
fibrous structure of the heart plays a major role when considering 
the electrical and mechanical properties of the heart. And while 
the fibres have a complex and generally well\hyp{}organised structure,
see e.g.\ \cite{DEHPR18,RSR18,RSG07,Setal12},
the exact fibre orientation may vary between individuals and also over time in an individual,
for example due to the presence of scaring of the heart.

More generally, we wish to be able to model diffusion in a fibrous media,
where fibre direction and diffusion strength in fibre direction are subject to uncertainty.
For this setting, the following random anisotropic diffusion coefficient
was defined in \cite{HPS17b}:
\begin{equation*}
  \Abfm[\omega] \isdef a \Ibfm + \groupp[\Big]{\norm[\big]{\Vbfm[\omega]}_2 - a} \frac{\Vbfm[\omega] \Vbfm^\trans[\omega]}{\Vbfm^\trans[\omega] \Vbfm[\omega]},
\end{equation*}
where $a$ is a given value and $\Vbfm$ is a random vector\hyp{}valued field,
over a given spatial domain $D$ and a given probability space $(\Omega, \Fcal, \Pbbb)$.
The fibre direction is hence given by $\Vbfm / \norm{\Vbfm}_2$ with
the diffusion strength in the fibre direction being $\norm{\Vbfm}_2$
and the diffusion strength perpendicular to the fibre direction is defined by $a$.
While we only consider this model hereafter,
the techniques we use may also be applied straightforwardly to
other models of random anisotropic diffusion coefficients,
such as, for example, the following model in three spatial dimensions:
\begin{align*}
  \Abfm[\omega] \isdef&
  \begin{bmatrix} | & | & | \\ \fbfm & \tbfm & \sbfm \\ | & | & | \end{bmatrix}
  \begin{bmatrix}
    \cos \alpha_3[\omega] & -\sin \alpha_3[\omega] & 0 \\
    \sin \alpha_3[\omega] & \cos \alpha_3[\omega] & 0 \\
    0 & 0 & 1
  \end{bmatrix}
  \begin{bmatrix}
    1 & 0 & 0 \\
    0 & \cos \alpha_2[\omega] & \sin \alpha_2[\omega] \\
    0 & -\sin \alpha_2[\omega] & \cos \alpha_2[\omega]
  \end{bmatrix} \\
  & \begin{bmatrix}
    \cos \alpha_1[\omega] & 0 & \sin \alpha_1[\omega] \\
    0 & 1 & 0 \\
    -\sin \alpha_1[\omega] & 0 & \cos \alpha_1[\omega]
  \end{bmatrix}
  \begin{bmatrix}
    a_\fbfm[\omega] & 0 & 0 \\
    0 & a_\tbfm[\omega] & 0 \\
    0 & 0 & a_\sbfm[\omega]
  \end{bmatrix}
  \begin{bmatrix}
    \cos \alpha_1[\omega] & 0 & \sin \alpha_1[\omega] \\
    0 & 1 & 0 \\
    -\sin \alpha_1[\omega] & 0 & \cos \alpha_1[\omega]
  \end{bmatrix}^\trans \\
  & \begin{bmatrix}
    1 & 0 & 0 \\
    0 & \cos \alpha_2[\omega] & \sin \alpha_2[\omega] \\
    0 & -\sin \alpha_2[\omega] & \cos \alpha_2[\omega]
  \end{bmatrix}^\trans
  \begin{bmatrix}
    \cos \alpha_3[\omega] & -\sin \alpha_3[\omega] & 0 \\
    \sin \alpha_3[\omega] & \cos \alpha_3[\omega] & 0 \\
    0 & 0 & 1
  \end{bmatrix}^\trans
  \begin{bmatrix} | & | & | \\ \fbfm & \tbfm & \sbfm \\ | & | & | \end{bmatrix}^\trans
\end{align*}
Here, $\fbfm$, $\tbfm$ and $\sbfm$ are vector fields describing the fibre,
the transverse sheet and the sheet normal directions in the heart,
which at each point in $D$ yields an orthonormal basis.
These vector fields could, for example, be derived from measurements
or be generated by an algorithm such as
the Laplace\hyp{}Dirichlet Rule\hyp{}Based algorithm described in \cite{BBPT12}.
Note that, in this model, the diffusion strengths are random fields
$a_\fbfm$, $a_\tbfm$ and $a_\sbfm$,
and the fibre and sheet directions are locally angularily perturbed by random fields
$\alpha_1$, $\alpha_2$ and $\alpha_3$.

We shall consider the second order diffusion problem with this uncertain
diffusion coefficient $\Abfm$ given by
\begin{equation*}
  \text{for almost every $\omega \in \Omega$: } \left\{
    \begin{alignedat}{2}
      - \Div_\xbfm \groupp[\big]{\Abfm[\omega] \Grad_\xbfm u[\omega]} &= f
      &\quad &\text{in }D , \\
      u[\omega] &= 0
      &\quad &\text{on }\partial D ,
    \end{alignedat}
  \right.
\end{equation*}
with the known function $f$ as a source. The result of this article is 
then as follows. Having spatial $H^s$-regularity of the underlying diffusion 
problem, given by sufficient smoothness of the right hand $f$ side and the domain $D$,
then the random solution $u$ admits analytic regularity 
with respect to the stochastic parameter also in the $H^s(D)$-norm provided 
that the random vector-valued field offers enough spatial regularity. This 
\emph{mixed} regularity is the essential ingredient in order to apply 
multilevel collocation or multilevel quadrature methods without deteriorating 
the rate of convergence, see \cite{HPS13} for instance.

The rest of the article is organised as follows: In Section 2,
we provide basic definitions and notation for the functional analytic 
framework to be able to state and then also reformulate the model 
problem, by using the Karhunen\hyp{}Lo\`eve expansion of the 
diffusion describing random vector\hyp{}valued field $\Vbfm$,
into its stochastically parametric and spatially weak formulation.
Section 3 then deals with the regularity of the solution
of the stochastically parametric and spatially weak formulation of the model problem
with respect to the stochastic parameter and some given higher spatial regularity
in the model problem.
We then use the fact
that the higher spatial regularity can be kept,
when considering the regularity of the solution with respect to the stochastic parameter,
to arrive at convergence rates when considering multilevel quadrature,
such as multilevel quasi\hyp{}Monte Carlo quadrature,
to approximate the solution's mean and second moment.
Numerical examples are provided in Section 4 as validation;
specifically we use multilevel quasi\hyp{}Monte Carlo quadrature
to approximate the solution's mean and second moment
in a setting with three spatial dimensions.
Lastly, we give our conclusions in Section 5.
%
\section{Problem formulation}
\subsection{Notation and precursory remarks}
For a given Banach space $\Xcal$ and a complete measure 
space $\Mcal$ with measure $\mu$ the space $L_\mu^p(\Mcal;\Xcal)$ 
for $1 \leq p \leq \infty$ denotes the Bochner space, see \cite{HillePhillips},
which contains all equivalence classes of strongly measurable functions
$v \colon \Mcal \to \Xcal$ with finite norm
\begin{equation*}
  \norm{v}_{L_\mu^p(\Mcal; \Xcal)} \isdef
  \begin{cases}
    \groups[\bigg]{\displaystyle\int_\Mcal \norm[\big]{v(x)}_{\Xcal}^p \dif\,\mu(x)}^{1/p} , & p < \infty , \\
    \displaystyle\esssup_{x \in \Mcal} \norm[\big]{v(x)}_{\Xcal} , & p = \infty .
  \end{cases}
\end{equation*}
A function $v \colon \Mcal \to \Xcal$ is strongly measurable if there exists 
a sequence of countably\hyp{}valued measurable functions $v_n \colon \Mcal \to \Xcal$,
such that for almost every $m \in \Mcal$ we have $\lim_{n \to \infty} v_n(m) = v(m)$.
Note that, for finite measures $\mu$, we also have the usual inclusion
$L_\mu^p(\Mcal; \Xcal) \supset L_\mu^q(\Mcal; \Xcal)$ for $1 \leq p < q \leq \infty$.

Let $\Xcal$, $\Xcal_1, \ldots, \Xcal_r$ and $\Ycal$ be Banach spaces,
then we denote the Banach space of bounded, linear maps from $\Xcal$ to $\Ycal$
as $\Bcal(\Xcal; \Ycal)$;
furthermore,
we recursively define
\begin{equation*}
  \Bcal(\Xcal_1, \ldots, \Xcal_r; \Ycal) \isdef \Bcal\groupp[\big]{\Xcal_1; \Bcal(\Xcal_2, \ldots, \Xcal_r; \Ycal)}
\end{equation*}
and the special case
\begin{equation*}
  \Bcal^0(\Xcal; \Ycal) \isdef \Ycal
  \quad\text{and}\quad
  \Bcal^{r+1}(\Xcal; \Ycal) \isdef \Bcal\groupp[\big]{\Xcal; \Bcal^r(\Xcal; \Ycal)} .
\end{equation*}
For $\Tbfm \in \Bcal(\Xcal_1, \ldots, \Xcal_r; \Ycal)$ and $\vbfm_j \in \Xcal_j$
we use the notation
$\Tbfm \vbfm_1 \cdots \vbfm_r \isdef \Tbfm(\vbfm_1, \ldots, \vbfm_r) \in \Ycal$.

Subsequently, we will always equip $\Rbbb^d$ with the norm $\norm{\cdot}_2$
induced by the canonical inner product $\langle\cdot,\cdot\rangle$
and $\Rbbb^{d \times d}$ with the induced norm $\norm{\cdot}_2$.
Then, for $\vbfm, \wbfm \in \Rbbb^d$, the Cauchy\hyp{}Schwartz inequality gives us
\begin{equation*}
  \norms{\vbfm^\trans \wbfm} = \norms[\big]{\groupa{\vbfm, \wbfm}} \leq \norm{\vbfm}_2 \norm{\wbfm}_2,
\end{equation*}
where equality holds only if $\vbfm = \wbfm$,
and we also have, by straightforward computation, that
\begin{equation*}
  \norm{\vbfm \wbfm^\trans}_2 = \norm{\vbfm}_2 \norm{\wbfm}_2 .
\end{equation*}

We also note that to avoid the use of generic but unspecified constants in
certain formulas we use $C \lesssim D$ to mean that $C$ can be bounded by
a multiple of $D$, independently of parameters which $C$ and $D$ may depend on.
Obviously, $C \gtrsim D$ is defined as $D \lesssim  C$ and we write
$C \eqsim  D$ if $C \lesssim D$ and $C \gtrsim  D$.
Lastly, note that for the natural numbers $\Nbbb$ denotes them including $0$ and
$\Nbbb^*$ excluding $0$.
\subsection{Model problem}
Let $(\Omega, \Fcal, \Pbbb)$ be a separable, complete probability space.
Then, we consider the following second order diffusion problem
with a random anisotropic diffusion coefficient
\begin{equation}
  \label{eq:sodp}
  \text{for almost every $\omega \in \Omega$: } \left\{
    \begin{alignedat}{2}
      - \Div_\xbfm \groupp[\big]{\Abfm[\omega] \Grad_\xbfm u[\omega]} &= f
      &\quad &\text{in }D , \\
      u[\omega] &= 0 &\quad &\text{on }\partial D ,
    \end{alignedat}
  \right.
\end{equation}
where $D \subset \Rbbb^d$ is a Lipschitz domain with $d \geq 1$
and the function $f \in H^{-1}(D)$ describes the known source.
The random anisotropic diffusion coefficient is given as the random matrix field
$\Abfm \in L_\Pbbb^\infty\groupp[\big]{\Omega; L^\infty(D; \Rbbb^{d \times d})}$,
which satisfies the uniform ellipticity condition
\begin{equation}
  \label{eq:Aellipticity}
  \underline{a} \leq \essinf_{\xbfm \in D} \lambda_{\min} \groupp[\big]{\Abfm[\omega](\xbfm)} \leq
  \esssup_{\xbfm \in D} \lambda_{\max} \groupp[\big]{\Abfm[\omega](\xbfm)} \leq \overline{a}
  \quad\text{$\Pbbb$-almost surely}
\end{equation}
for some constants $0 < \underline{a} \leq \overline{a} < \infty$
and is almost surely symmetric almost everywhere.
Without loss of generality, we assume $\underline{a} \leq 1 \leq \overline{a}$.

We specifically consider diffusion coefficients that are of form
\begin{equation}
  \label{eq:AV}
  \Abfm[\omega](\xbfm) \isdef a \Ibfm + \groupp[\Big]{\norm[\big]{\Vbfm[\omega](\xbfm)}_2 - a} \frac{\Vbfm[\omega](\xbfm) \Vbfm^\trans[\omega](\xbfm)}{\Vbfm^\trans[\omega](\xbfm) \Vbfm[\omega](\xbfm)} ,
\end{equation}
where $a \in \Rbbb$ is a given positive number and
$\Vbfm \in L_\Pbbb^\infty\groupp[\big]{\Omega; L^\infty(D; \Rbbb^d)}$
is a random vector\hyp{}valued field.
We note that such a field $\Abfm$ accounts for a medium that has 
homogeneous diffusion strength $a$ perpendicular to $\Vbfm$
and has diffusion strength $\norm[\big]{\Vbfm[\omega](\xbfm)}_2$
in the direction of $\Vbfm$.
The randomness of the specific direction and length of $\Vbfm$
therefore quantifies uncertainty of this notable direction and its diffusion strength.
To guarantee the uniform ellipticity condition \eqref{eq:Aellipticity},
we require that
\begin{equation}
  \label{eq:Vellipticity}
  \underline{a} \leq \essinf_{\xbfm \in D} \norm[\big]{\Vbfm[\omega](\xbfm)}
  \leq \esssup_{\xbfm \in D} \norm[\big]{\Vbfm[\omega](\xbfm)} \leq \overline{a}
  \quad\text{$\Pbbb$-almost surely}
\end{equation}
as well as $\underline{a} \leq a \leq \overline{a}$.

It is assumed that the spatial variable $\xbfm$ and the stochastic 
parameter $\omega$ of the random field have been separated by 
the Karhunen\hyp{}Lo\`eve expansion of $\Vbfm$, yielding a 
parametrised expansion
\begin{equation}
  \label{eq:KLp}
  \Vbfm[\ybfm](\xbfm)
  = \Mean[\Vbfm](\xbfm) + \sum_{k=1}^{\infty} \sigma_k \psibfm_k(\xbfm) y_k ,
\end{equation}
where $\ybfm = (y_k)_{k \in \Nbbb^*} \in \square \isdef \groups{{-1}, 1}^{\Nbbb^*}$ is
a sequence of uncorrelated random variables, see e.g.\ \cite{HPS17b}.
In the following, we will denote the pushforward of the measure $\Pbbb$
onto $\square$ as $\Pbbb_\ybfm$.
Then, we also view $\Abfm[\ybfm](\xbfm)$ and $u[\ybfm](\xbfm)$ 
as being parametrised by $\ybfm$ and restate \eqref{eq:sodp} as
\begin{equation}
  \label{eq:psodp}
  \text{for almost every $\ybfm \in \square$: } \left\{
    \begin{alignedat}{2}
      - \Div_\xbfm \groupp[\big]{\Abfm[\ybfm] \Grad_\xbfm u[\ybfm]} &= f
      &\quad &\text{in }D , \\
      u[\ybfm] &= 0 &\quad &\text{on }\partial D.
    \end{alignedat}
  \right.
\end{equation}

We now impose some common assumptions,
which make the Karhunen\hyp{}Lo\`eve expansion computationally feasible.
\begin{assumption}
  The random variables $(y_k)_{k \in \Nbbb^*}$ are independent and identically distributed.
  Moreover, they are uniformly distributed on $\groups[\big]{{-1}, 1}$.
\end{assumption}

Lastly, we note that the spatially weak form of \eqref{eq:psodp} is given by
\begin{equation}
  \label{eq:swpsodp}
  \left\{
  \begin{aligned}
    & \text{Find $u \in L_{\Pbbb_\ybfm}^\infty\groupp[\big]{\square; H_0^1(D)}$ such that} \\
    & \quad \groupp[\big]{\Abfm[\ybfm] \Grad_\xbfm u[\ybfm], \Grad_\xbfm v}_{L^2(D; \Rbbb^d)} = \groupp[\big]{f, v}_{L^2(D; \Rbbb^d)} \\
    & \quad \text{for almost every $\ybfm \in \square$ and all $v \in H_0^1(D)$.}
  \end{aligned}
  \right.
\end{equation}
This also entails the well known stability estimate.
\begin{lemma}\label{lemma:ubound}
  There is a unique solution $u \in L_{\Pbbb_\ybfm}^\infty\groupp[\big]{\square; H_0^1(D)}$
  of \eqref{eq:swpsodp}, which fulfils
  \begin{equation*}
    \norm[\big]{u[\ybfm]}_{L_{\Pbbb_\ybfm}^\infty(\square; H^1(D))}
    \leq \frac{1}{\underline{a} c_V^2} \groupp[\Big]{\norm{f}_{H^{-1}(D)}} ,
  \end{equation*}
  where $c_V$ is the Poincar\'e-Friedrichs constant of $H_0^1(D)$.
\end{lemma}
%
\section{Parametric regularity and multilevel quadrature}
We now derive regularity estimates for the solution $u$ of \eqref{eq:swpsodp}
and apply multilevel quadrature for approximating the mean of $u$.
The regularity estimates are based on
the following assumption on the decay of the expansion of $\Vbfm$.
\begin{assumption}\label{assumption:Vdecay}
  We assume that the $\psibfm_k$ are elements of $W^{\kappa, \infty}(D; \Rbbb^d)$
  for a $\kappa \in \Nbbb$
  and that the sequence $\gammabfm_\kappa = \groupp{\gamma_{\kappa, k}}_{k \in \Nbbb}$,
  given by
  \begin{equation*}
    \gamma_{\kappa, k} \isdef \norm[\big]{\sigma_k \psibfm_k}_{W^{\kappa, \infty}(D; \Rbbb^d)} ,
  \end{equation*}
  is at least in $\ell^1(\Nbbb)$,
  where we have defined $\psibfm_0 \isdef \Mean[\Vbfm]$ and
  $\sigma_0 \isdef 1$.
  Furthermore, we define
  \begin{equation*}
    c_{\gammabfm_\kappa} = \max\groupb[\big]{\norm{\gammabfm_\kappa}_{\ell^1(\Nbbb)}, 2} .
  \end{equation*}

  We furthermore assume that the vector field $\Vbfm$ is given by a finite rank
  Karhunen-Lo\`eve expansion, i.e.
  \begin{equation*}
    \Vbfm[\ybfm](\xbfm)
    = \Mean[\Vbfm](\xbfm) + \sum_{k=1}^{M} \sigma_k \psibfm_k(\xbfm) y_k ,
  \end{equation*}
  where $\square \isdef \groups{{-1}, 1}^{M}$.
  We note that the regularity estimates however will not depend on the rank $M$
  and therefore,
  if necessary,
  a finite rank can be attained by appropriate truncation.
\end{assumption}
Furthermore,
for the regularity estimates we also require an elliptic regularity result.
\begin{assumption}\label{assumption:er}
  Let $\Rcal_\kappa$ be a Banach space with norm $\norm{\cdot}_{\Rcal_\kappa}$ such that,
  for all
  \begin{equation*}
    \Abfm \in W^{\kappa, \infty}(D; \Rbbb^{d \times d}) \cap \Rcal_\kappa
  \end{equation*}
  that fulfil \eqref{eq:Aellipticity},
  we have that
  the problem of solving
  \begin{equation*}
    \groupp[\Big]{\Abfm \Grad_{\xbfm} u, \Grad_x v}_{L^2(D; \Rbbb^d)} = \groupp{h, v}_{L^2(D)}
  \end{equation*}
  for any $h \in H^{\kappa-1}(D)$ has a unique solution $u \in H_0^1(D)$,
  which also lies in $H^{\kappa+1}(D)$,
  with
  \begin{equation*}
    \norm{u}_{\kappa+1, 2, D}
    \leq C_{\kappa, er}\groupp[\big]{D, \norm{\Abfm}_{\Rcal_\kappa}} \norm{h}_{\kappa-1, 2, D} ,
  \end{equation*}
  where $C_{\kappa, er}$ only depends on $D$ and continuously on $\norm{\Abfm}_{\Rcal_\kappa}$.

  We assume from here on that
  $\Abfm$ also lies in $L_{\Pbbb_\ybfm}^\infty(\square; \Rcal_\kappa)$.
\end{assumption}
Note, that for $\kappa = 0$, this reduces to the stability estimate,
for which the parametric regularity may be found in \cite{HPS17b}.
Therefore, we will hereafter only consider the case where $\kappa \geq 1$.
Such an elliptic regularity estimate for example is known for $\kappa = 1$,
when the domain is convex and bounded and
$\Rcal_\kappa = C^{0, 1}(\overline{D}; \Rbbb^{d \times d})$,
see \cite[Theorems 3.2.1.2 and 3.1.3.1]{Grisvard}.
The elliptic regularity estimate is also known to hold for $\kappa \geq 1$ and $d = 2$,
when the domain's boundary is smooth and
$\Rcal_\kappa = W^{\kappa, \infty}(D; \Rbbb^{d \times d})$,
see \cite{BLN17}.

The rest of this section is now split into three subsections.
The first subsection is dedicated to introducing
some useful norms as well as deriving Corollary~\ref{corollary:prodrules22b}
and Theorem~\ref{theorem:comprules}.
These two results are then used in the second subsection to step\hyp{}by\hyp{}step
provide regularity estimates for the different terms
that make up the diffusion coefficient,
based on Assumption~\ref{assumption:Vdecay} on the decay of the expansion of $\Vbfm$,
which then yields regularity estimates for the diffusion coefficient.
By using this decay as well as Assumption~\ref{assumption:er}, we derive
the regularity estimates of the solution $u$ in Theorem~\ref{theorem:euybounds}.
In the third subsection, we then briefly discuss what kind of convergence rates
and computational complexity this regularity of $u$ yields,
for the example of approximating $\Mean[u]$ using multilevel quadrature methods.
\subsection{Precursory remarks}
For the Sobolev--Bochner spaces $W^{\eta, p}(D; \Xcal)$
with $\eta \in \Nbbb$ and $1 \leq p \leq \infty$,
we introduce the norms given by
\begin{equation*}
  \norm{v}_{\eta, p, D; \Xcal}
  \isdef \norm{v}_{W^{\eta, p}(D; \Xcal)}
  \isdef
  \sum_{\norms{\alphabfm} \leq \eta} \frac{1}{\alphabfm!}
  \norm[\big]{\pdif_\xbfm^\alphabfm v}_{p, D; \Xcal} ,
\end{equation*}
for $v \in W^{\eta, p}(D; \Xcal)$
where $\Xcal$ is a Banach space with norm $\norm{\cdot}_{\Xcal}$
and where we make use of the shorthand
\begin{equation*}
  \norm{\cdot}_{p, D; \Xcal}
  \isdef \norm{\cdot}_{L^p(D; \Xcal)} .
\end{equation*}
We may omit the specification of the Banach space $\Xcal$,
for example when $\Xcal$ is the space $\Rbbb$, $\Rbbb^d$ or $\Rbbb^{d \times d}$.

For these norms,
we have the following lemma giving a bound after applying
a $\Div_\xbfm$, $\Dif_\xbfm$ or $\Grad_\xbfm$:
\begin{lemma}\label{lemma:divjacgradrule}
  Let $\eta \in \Nbbb^*$, $1 \leq p \leq \infty$.
  For $\vbfm \in W^{\eta, p}(D; \Rbbb^{d})$
  we have that $\Div_\xbfm \vbfm \in W^{\eta-1, p}(D; \Rbbb)$ with
  \begin{equation*}
    \norm{\Div_\xbfm \vbfm}_{\eta-1, p, D}
    \leq \eta d \norm{\vbfm}_{\eta, p, D}
  \end{equation*}
  and $\Dif_\xbfm \vbfm \in W^{\eta-1, p}(D; \Rbbb^{d \times d})$ with
  \begin{equation*}
    \norm{\Dif_\xbfm \vbfm}_{\eta-1, p, D}
    \leq \eta d \norm{\vbfm}_{\eta, p, D} .
  \end{equation*}
  For $v \in W^{\eta, p}(D)$
  we have that $\Grad_\xbfm v \in W^{\eta-1, p}(D; \Rbbb^{d})$ with
  \begin{equation*}
    \norm{\Grad_\xbfm v}_{\eta-1, p, D}
    \leq \eta d \norm{v}_{\eta, p, D} .
  \end{equation*}
\end{lemma}
\begin{proof}
  We calculate
  \begin{align*}
    \norm{\Div_\xbfm \vbfm}_{\eta-1, p, D}
    &= \norm[\bigg]{\sum_{i=1}^{d} \pdif_{x_i} v_i}_{\eta-1, p, D} 
    = \sum_{\norms{\alphabfm} \leq \eta-1} \frac{1}{\alphabfm!}
    \norm[\bigg]{\sum_{i=1}^{d} \pdif_\xbfm^{\alphabfm}\pdif_{x_i} v_i}_{p, D} \\
    &\leq \sum_{\norms{\alphabfm} \leq \eta-1} \frac{1}{\alphabfm!}
    \sum_{i=1}^{d} \norm{\pdif_\xbfm^{\alphabfm}\pdif_{x_i} v_i}_{p, D} 
    \leq \eta \sum_{\norms{\alphabfm} \leq \eta} \frac{1}{\alphabfm!}
    \sum_{i=1}^{d} \norm{\pdif_\xbfm^{\alphabfm} v_i}_{p, D} \\
    &\leq \eta d \sum_{\norms{\alphabfm} \leq \eta} \frac{1}{\alphabfm!}
    \norm{\pdif_\xbfm^{\alphabfm} \vbfm}_{p, D} 
    \leq \eta d \norm{\vbfm}_{\eta, p, D} .
  \end{align*}
  
  We also may compute
  \begin{align*}
    \norm{\Dif_\xbfm \vbfm}_{\eta-1, p, D}
    &= \sum_{\norms{\alphabfm} \leq \eta-1} \frac{1}{\alphabfm!}
    \norm[\big]{\pdif_\xbfm^\alphabfm \Dif_\xbfm \vbfm}_{p, D} 
    = \sum_{\norms{\alphabfm} \leq \eta-1} \frac{1}{\alphabfm!}
    \norm*{\begin{bmatrix}
        \pdif_\xbfm^\alphabfm \pdif_{x_1} \vbfm &
        \hdots &
        \pdif_\xbfm^\alphabfm \pdif_{x_d} \vbfm
    \end{bmatrix}}_{p, D} \\
    &\leq \sum_{\norms{\alphabfm} \leq \eta-1} \frac{1}{\alphabfm!}
    \sum_{i=1}^{d} \norm{\pdif_\xbfm^\alphabfm \pdif_{x_i} \vbfm}_{p, D} 
    \leq \eta \sum_{\norms{\alphabfm} \leq \eta} \frac{1}{\alphabfm!}
    \sum_{i=1}^{d} \norm{\pdif_\xbfm^\alphabfm \vbfm}_{p, D} \\
    &\leq \eta d \sum_{\norms{\alphabfm} \leq \eta} \frac{1}{\alphabfm!}
    \norm{\pdif_\xbfm^\alphabfm \vbfm}_{p, D} 
    \leq \eta d \norm{\vbfm}_{\eta, p, D} 
  \end{align*}
  as well as
  \begin{align*}
    \norm{\Grad_\xbfm u}_{\eta-1, p, D}
    \norm{\Dif_\xbfm u}_{\eta-1, p, D}
    &= \sum_{\norms{\alphabfm} \leq \eta-1} \frac{1}{\alphabfm!}
    \norm[\big]{\pdif_\xbfm^\alphabfm \Grad_\xbfm u}_{p, D} 
    = \sum_{\norms{\alphabfm} \leq \eta-1} \frac{1}{\alphabfm!}
    \norm*{\begin{bmatrix}
        \pdif_\xbfm^\alphabfm \pdif_{x_1} u \\
        \vdots \\
        \pdif_\xbfm^\alphabfm \pdif_{x_d} u
    \end{bmatrix}}_{p, D} \\
    &\leq \sum_{\norms{\alphabfm} \leq \eta-1} \frac{1}{\alphabfm!}
    \sum_{i=1}^{d} \norm{\pdif_\xbfm^\alphabfm \pdif_{x_i} u}_{p, D} 
    \leq \eta \sum_{\norms{\alphabfm} \leq \eta} \frac{1}{\alphabfm!}
    \sum_{i=1}^{d} \norm{\pdif_\xbfm^\alphabfm u}_{p, D} \\
    &\leq \eta d \sum_{\norms{\alphabfm} \leq \eta} \frac{1}{\alphabfm!}
    \norm{\pdif_\xbfm^\alphabfm u}_{p, D} 
    \leq \eta d \norm{u}_{\eta, p, D} . \qedhere
  \end{align*}
\end{proof}

The Leibniz rule also yields the following kind of submultiplicativity
for these norms:
\begin{lemma}\label{lemma:prodrule}
  Let $\eta \in \Nbbb$, $1 \leq p, p' \leq \infty$,
  $\Xcal$ and $ \Ycal$ be Banach spaces
  and
  \begin{equation*}
    \Mbfm \in W^{\eta, p}\groupp[\big]{D; \Bcal(\Xcal; \Ycal)} , \quad
    \vbfm \in W^{\eta, p'}(D; \Xcal)
  \end{equation*}
  with $q = (p^{-1} + p'^{-1})^{-1} \geq 1$.
  Then, we have
  \begin{equation*}
    \norm{\Mbfm \vbfm}_{\eta, q, D; \Ycal}
    \leq \norm{\Mbfm}_{\eta, p, D; \Bcal(\Xcal; \Ycal)}
    \norm{\vbfm}_{\eta, p', D; \Xcal} .
  \end{equation*}
\end{lemma}
\begin{proof}
  Let $\alphabfm, \betabfm \in \Nbbb^d$ be two multi\hyp{}indices,
  then we have
  \begin{equation*}
    \frac{1}{(\alphabfm + \betabfm)!}
    \binom{\alphabfm + \betabfm}{\betabfm}
    = \frac{1}{(\alphabfm + \betabfm)!}
    \frac{(\alphabfm + \betabfm)!}{\alphabfm! \betabfm!}
    = \frac{1}{\alphabfm! \betabfm!} .
  \end{equation*}
  
  We now can calculate
  \begin{align*}
    \norm{\Mbfm \vbfm}_{\eta, q, D; \Ycal}
    &= \sum_{\norms{\alphabfm} \leq \eta} \frac{1}{\alphabfm!}
    \norm[\big]{\pdif_\xbfm^\alphabfm [\Mbfm \vbfm]}_{q, D; \Ycal} \\
    &= \sum_{\norms{\alphabfm} \leq \eta} \frac{1}{\alphabfm!}
    \norm[\Bigg]{\sum_{\betabfm \leq \alphabfm} \binom{\alphabfm}{\betabfm}
      \pdif_\xbfm^\betabfm \Mbfm
      \pdif_\xbfm^{\alphabfm-\betabfm} \vbfm}_{q, D; \Ycal} \\
    &\leq \sum_{\norms{\alphabfm} \leq \eta} \sum_{\betabfm \leq \alphabfm}
    \frac{1}{\alphabfm!} \binom{\alphabfm}{\betabfm}
    \norm[\big]{\pdif_\xbfm^\betabfm \Mbfm
      \pdif_\xbfm^{\alphabfm-\betabfm} \vbfm}_{q, D; \Ycal} \\
    &\leq \sum_{\norms{\alphabfm} \leq \eta} \sum_{\betabfm \leq \alphabfm}
    \frac{1}{\alphabfm!} \binom{\alphabfm}{\betabfm}
    \norm[\big]{\pdif_\xbfm^\betabfm \Mbfm}_{p, D; \Bcal(\Xcal; \Ycal))}
    \norm[\big]{\pdif_\xbfm^{\alphabfm-\betabfm} \vbfm}_{p', D; \Xcal} .
  \end{align*}
  By a change of variables,
  i.e.\ replacing $\alphabfm$ with $\alphabfm+\betabfm$,
  and remarking that
  \begin{equation*}
    \groupb[\big]{(\alphabfm-\betabfm, \betabfm) : \norms{\alphabfm} \leq \eta,\,\betabfm \leq \alphabfm}
    = \groupb[\big]{(\alphabfm, \betabfm) : \norms{\alphabfm} + \norms{\betabfm} \leq \eta},
  \end{equation*}    
  we find the identity
  \begin{align*}
    & \sum_{\norms{\alphabfm} \leq \eta} \sum_{\betabfm \leq \alphabfm}
    \frac{1}{\alphabfm!} \binom{\alphabfm}{\betabfm}
    \norm[\big]{\pdif_\xbfm^\betabfm \Mbfm}_{p, D; \Bcal(\Xcal; \Ycal)}
    \norm[\big]{\pdif_\xbfm^{\alphabfm-\betabfm} \vbfm}_{p', D; \Xcal} \\
    &\quad=\sum_{\norms{\alphabfm}+\norms{\betabfm} \leq \eta}
    \frac{1}{(\alphabfm + \betabfm)!}
    \binom{\alphabfm+\betabfm}{\betabfm}
    \norm[\big]{\pdif_\xbfm^\betabfm \Mbfm}_{p, D; \Bcal(\Xcal; \Ycal)}
    \norm[\big]{\pdif_\xbfm^{\alphabfm} \vbfm}_{p', D; \Xcal} \\
    &\quad= \sum_{\norms{\alphabfm}+\norms{\betabfm} \leq \eta}
    \frac{1}{\alphabfm! \betabfm!}
    \norm[\big]{\pdif_\xbfm^\betabfm \Mbfm}_{p, D; \Bcal(\Xcal; \Ycal)}
    \norm[\big]{\pdif_\xbfm^{\alphabfm} \vbfm}_{p', D; \Xcal} .
  \end{align*}
  Consequently,
  we arrive at the desired estimate:
  \begin{align*}
    \norm{\Mbfm \vbfm}_{\eta, q, D; \Ycal}
    &\leq \sum_{\norms{\alphabfm},\norms{\betabfm} \leq \eta}
    \frac{1}{\alphabfm! \betabfm!}
    \norm[\big]{\pdif_\xbfm^\betabfm \Mbfm}_{p, D; \Bcal(\Xcal; \Ycal)}
    \norm[\big]{\pdif_\xbfm^{\alphabfm} \vbfm}_{p', D; \Xcal} \\
    &\leq \sum_{\norms{\betabfm}\leq \eta} 
    \frac{1}{\betabfm!}
    \norm[\big]{\pdif_\xbfm^\betabfm \Mbfm}_{p, D; \Bcal(\Xcal; \Ycal)}
    \sum_{\norms{\alphabfm} \leq \eta}
    \frac{1}{\alphabfm!}
    \norm[\big]{\pdif_\xbfm^{\alphabfm} \vbfm}_{p', D; \Xcal} \\
    &= \norm{\Mbfm}_{\eta, p, D; \Bcal(\Xcal; \Ycal)}
    \norm{\vbfm}_{\eta, p', D; \Xcal} .
    \qedhere
  \end{align*}
\end{proof}

By induction, we thus arrive at the following corollary:
\begin{corollary}\label{corollary:prodrule1}
  Let $\eta \in \Nbbb$, $1 \leq p, p_1, \ldots, p_r \leq \infty$,
  $\Xcal_1, \ldots, \Xcal_r$ and $ \Ycal$ be Banach spaces
  and
  \begin{equation*}
    \Mbfm \in W^{\eta, p}\groupp[\big]{D; \Bcal(\Xcal_1, \ldots, \Xcal_r; \Ycal)} , \quad
    \vbfm_j \in W^{\eta, p_j}(D; \Xcal_j)
  \end{equation*}
  with $q = (p^{-1} + p_1^{-1} + \cdots + p_r^{-1})^{-1} \geq 1$.
  Then, we have
  \begin{equation*}
    \norm{\Mbfm \vbfm_1 \cdots \vbfm_r}_{\eta, q, D; \Ycal}
    \leq \norm{\Mbfm}_{\eta, p, D; \Bcal(\Xcal_1, \ldots, \Xcal_r; \Ycal)}
    \prod_{j=1}^r \norm{\vbfm_j}_{\eta, p_j, D; \Xcal_j} .
  \end{equation*}
\end{corollary}

As we will need the Fa\`a di Bruno formula, see \cite{CS96},
we just restate it here
--- slightly adapted to our notation and usage ---
for reference:
\begin{remark}\label{remark:faadibruno}
  Given $\Wbfm \colon X \to \Ycal$
  and $\vbfm \colon D \to \Xcal$,
  where $\Xcal$ and $\Ycal$ are Banach spaces,
  $X \subset \Xcal$ is open with $\Img_D \vbfm \subset X$
  and $\Wbfm, \vbfm$ are both sufficiently differentiable for the formula to make sense,
  then
  \begin{equation*}
    \pdif_\xbfm^\alphabfm (\Wbfm \circ \vbfm)
    = \alphabfm!
    \sum_{r=1}^{\norms{\alphabfm}}
    \frac{1}{r!}
    \sum_{P(\alphabfm, r)}
    (\Dif^r \Wbfm \circ \vbfm)
    \pdif_\xbfm^{\betabfm_1} \vbfm \cdots
    \pdif_\xbfm^{\betabfm_r} \vbfm
    \prod_{j=1}^r \frac{1}{\betabfm_j!} ,
  \end{equation*}
  for $\alphabfm \in \Nbbb^M$ with $\alphabfm \neq \zerobfm$,
  where $P(\alphabfm, r)$ is the set of integer partitions of a multiindex $\alphabfm$
  into $r$ non\hyp{}vanishing multiindices, given by
  \begin{equation*}
    P(\alphabfm, r) \isdef
    \groupb[\bigg]{
    \groupp[\big]{\betabfm_1, \ldots, \betabfm_r}
    \in \groupp[\big]{\Nbbb^M}^r :
    \sum_{j=1}^r \betabfm_j = \alphabfm \text{ and } \betabfm_{j} \neq \zerobfm \text{ for all } 1 \leq j \leq r
    } .
  \end{equation*}

  It also follows from \cite{CS96} that,
  for $\alphabfm \in \Nbbb^d$ and $r \in \Nbbb$,
  we have
  \begin{equation*}
    \alphabfm! \sum_{P(\alphabfm, r)} \prod_{j=1}^{r} \frac{1}{\betabfm_j!} = r! S_{\norms{\alphabfm},r} ,
  \end{equation*}
  where $S_{n,r}$ denotes the Stirling numbers of the second kind, see \cite{AbramowitzStegun}.

  Lastly, as we know that
  $\sum_{r=1}^{\norms{\alphabfm}} r! S_{\norms{\alphabfm}, r}$ equals
  the $\norms{\alphabfm}$-th ordered Bell number,
  we can bound it,
  see \cite{BTNT12},
  giving
  \begin{equation*}
    \alphabfm! \sum_{r=1}^{\norms{\alphabfm}} \sum_{P(\alphabfm, r)} \prod_{j=1}^{r} \frac{1}{\betabfm_j!}
    = \sum_{r=1}^{\norms{\alphabfm}} r! S_{\norms{\alphabfm}, r}
    \leq \frac{\norms{\alphabfm}!}{(\log 2)^{\norms{\alphabfm}}} .
  \end{equation*}
\end{remark}

Lastly, the Fa\`a di Bruno formula now yields the following lemma:
\begin{lemma}\label{lemma:comprule}
  Let $\eta \in \Nbbb$, $1 \leq p \leq \infty$,
  $\Xcal$ and $\Ycal$ be Banach spaces,
  \begin{equation*}
    \vbfm \in W^{\eta, \infty}(D; \Xcal) ,
  \end{equation*}
  $X \subset \Xcal$ be open with $\Img_D \vbfm \subset X$,
  $\Wbfm \colon X \to \Ycal$ $\eta$-times Fr\'echet differentiable
  and, for $0 \leq r \leq \eta$,
  \begin{equation*}
    \Dif^r \Wbfm \circ \vbfm \in L^{p}\groupp[\big]{D; \Bcal^{r}(\Xcal; \Ycal)} .
  \end{equation*}
  Then, we have
  \begin{equation*}
    \Wbfm \circ \vbfm \in W^{\eta, p}(D; \Ycal)
  \end{equation*}
  with
  \begin{equation*}
    \norm{\Wbfm \circ \vbfm}_{\eta, p, D; \Ycal}
    \leq \sum_{r=0}^{\eta}
    \frac{1}{r!} \norm{\Dif^r \Wbfm \circ \vbfm}_{p, D; \Bcal^{r}(\Xcal; \Ycal)}
    \norm{\vbfm}_{\eta, \infty, D; \Xcal}^r .
  \end{equation*}
\end{lemma}
\begin{proof}
  The Fa\`a di Bruno formula leads to
  \begin{align*}
    &\norm{\Wbfm \circ \vbfm}_{\eta, p, D; \Ycal}
    = \sum_{\norms{\alphabfm} \leq \eta}
      \frac{1}{\alphabfm!}
      \norm[\big]{\pdif_\xbfm^\alphabfm (\Wbfm \circ \vbfm)}_{p, D; \Ycal} \\
    &\quad \leq \norm{\Wbfm \circ \vbfm}_{p, D; \Ycal} +
      \sum_{1 \leq \norms{\alphabfm} \leq \eta}
      \sum_{r=1}^{\norms{\alphabfm}} \frac{1}{r!}
      \sum_{P(\alphabfm, r)}
      \norm[\bigg]{
        (\Dif^r \Wbfm \circ \vbfm)
        \pdif_\xbfm^{\betabfm_1} \vbfm \cdots
        \pdif_\xbfm^{\betabfm_r} \vbfm
        \prod_{j=1}^r \frac{1}{\betabfm_j!}}_{p, D; \Ycal}
  \end{align*}
  Now, with Corollary~\ref{corollary:prodrule1},
  \begin{equation*}
    \norm[\bigg]{
      (\Dif^r \Wbfm \circ \vbfm)
      \pdif_\xbfm^{\betabfm_1} \vbfm \cdots
      \pdif_\xbfm^{\betabfm_r} \vbfm
      \prod_{j=1}^r \frac{1}{\betabfm_j!}}_{p, D; \Ycal}
    \leq \norm{\Dif^r \Wbfm \circ \vbfm}_{p, D; \Bcal^{r}(\Xcal; \Ycal)}
      \prod_{j=1}^{r}
      \frac{1}{\betabfm_j!} \norm[\big]{\pdif_\xbfm^{\betabfm_j} \vbfm}_{\infty, D; \Xcal}
  \end{equation*}
  leads to
  \begin{align*}
    &\norm{\Wbfm \circ \vbfm}_{\eta, p, D; \Ycal} \\
    &\quad \leq \norm{\Wbfm \circ \vbfm}_{p, D; \Ycal} +
    \sum_{1 \leq \norms{\alphabfm} \leq \eta}
    \sum_{r=1}^{\norms{\alphabfm}}
    \frac{1}{r!} \norm{\Dif^r \Wbfm \circ \vbfm}_{p, D; \Bcal^{r}(\Xcal; \Ycal)}
    \sum_{P(\alphabfm, r)}
    \prod_{j=1}^{r}
    \frac{1}{\betabfm_j!} \norm[\big]{\pdif_\xbfm^{\betabfm_j} \vbfm}_{\infty, D; \Xcal} \\
    &\quad \leq \norm{\Wbfm \circ \vbfm}_{p, D; \Ycal} +
    \sum_{r=1}^{\eta}
    \frac{1}{r!} \norm{\Dif^r \Wbfm \circ \vbfm}_{p, D; \Bcal^{r}(\Xcal; \Ycal)}
    \sum_{r \leq \norms{\alphabfm} \leq \eta}
    \sum_{P(\alphabfm, r)}
    \prod_{j=1}^{r}
    \frac{1}{\betabfm_j!} \norm[\big]{\pdif_\xbfm^{\betabfm_j} \vbfm}_{\infty, D; \Xcal} \\
    &\quad \leq \norm{\Wbfm \circ \vbfm}_{p, D; \Ycal} +
    \sum_{r=1}^{\eta}
    \frac{1}{r!} \norm{\Dif^r \Wbfm \circ \vbfm}_{p, D; \Bcal^{r}(\Xcal; \Ycal)}
    \groupp[\bigg]{\sum_{\betabfm \leq \alphabfm}
      \frac{1}{\betabfm!} \norm[\big]{\pdif_\xbfm^{\betabfm} \vbfm}_{\infty, D; \Xcal}}^r \\
    &\quad \leq \sum_{r=0}^{\eta}
    \frac{1}{r!} \norm{\Dif^r \Wbfm \circ \vbfm}_{p, D; \Bcal^{r}(\Xcal; \Ycal)}
    \norm{\vbfm}_{\eta, \infty, D; \Xcal}^r . \qedhere
  \end{align*}
\end{proof}

Furthermore,
we also introduce the shorthand notations
\begin{align*}
  \normt{v}_{\eta, p, D; \Xcal}
  &\isdef \norm{v}_{L_{\Pbbb_\ybfm}^\infty(\square; W^{\eta, p}(D; \Xcal))} ,
\end{align*}
for $v \in L_{\Pbbb_\ybfm}^\infty\groupp[\big]{\square; W^{\eta, p}(D; \Xcal)}$
where $\Xcal$ is a Banach space with norm $\norm{\cdot}_{\Xcal}$.
We may also omit the specification of the Banach space $\Xcal$ in this shorthand,
for example when $\Xcal$ is the space $\Rbbb$, $\Rbbb^d$ or $\Rbbb^{d \times d}$.

With this notation we still carry over the previous results,
yielding the following lemmata and corollaries.
Lemma~\ref{lemma:divjacgradrule} directly transforms to:
\begin{corollary}\label{corollary:divjacgradrules}
  Let $\eta \in \Nbbb^*$, $1 \leq p \leq \infty$.
  For $\vbfm \in L_{\Pbbb_\ybfm}^\infty\groupp[\big]{\square; W^{\eta, p}(D; \Rbbb^{d})}$
  we have that
  $\Div_\xbfm \vbfm \in L_{\Pbbb_\ybfm}^\infty\groupp[\big]{\square; W^{\eta-1, p}(D; \Rbbb)}$
  with
  \begin{equation*}
    \normt{\Div_\xbfm \vbfm}_{\eta-1, p, D}
    \leq \eta d \normt{\vbfm}_{\eta, p, D}
  \end{equation*}
  and
  $\Dif_\xbfm \vbfm \in L_{\Pbbb_\ybfm}^\infty\groupp[\big]{\square; W^{\eta-1, p}(D; \Rbbb^{d \times d})}$
  with
  \begin{equation*}
    \normt{\Dif_\xbfm \vbfm}_{\eta-1, p, D}
    \leq \eta d \normt{\vbfm}_{\eta, p, D} .
  \end{equation*}
  For $v \in L_{\Pbbb_\ybfm}^\infty\groupp[\big]{\square; W^{\eta, p}(D)}$
  we have that
  $\Grad_\xbfm v \in L_{\Pbbb_\ybfm}^\infty\groupp[\big]{\square; W^{\eta-1, p}(D; \Rbbb^{d})}$
  with
  \begin{equation*}
    \normt{\Grad_\xbfm v}_{\eta-1, p, D}
    \leq \eta d \normt{v}_{\eta, p, D} .
  \end{equation*}
\end{corollary}

We use Lemma~\ref{lemma:prodrule} to arrive at the following result:
\begin{lemma}\label{lemma:prodrules}
  Let $\eta, \in \Nbbb$, $1 \leq p, p' \leq \infty$, $\nubfm \in \Nbbb^M$,
  $\Xcal$ and $ \Ycal$ be Banach spaces
  and
  \begin{equation*}
    \pdif_\ybfm^\alphabfm \Mbfm \in L_{\Pbbb_\ybfm}^\infty\groupp[\big]{\square; W^{\eta, p}\groupp[\big]{D; \Bcal(\Xcal; \Ycal)}} , \quad
    \pdif_\ybfm^\alphabfm \vbfm \in L_{\Pbbb_\ybfm}^\infty\groupp[\big]{\square; W^{\eta, p'}(D; \Xcal)}
  \end{equation*}
  for all $\alphabfm \leq \nubfm$
  with $q = (p^{-1} + p'^{-1})^{-1} \geq 1$.
  Then, we have
  \begin{equation*}
    \pdif_\ybfm^\alphabfm (\Mbfm \vbfm) \in L_{\Pbbb_\ybfm}^\infty\groupp[\big]{\square; W^{\eta, p}(D; \Ycal)}
  \end{equation*}
  with
  \begin{equation*}
    \normt[\big]{\pdif_\ybfm^\alphabfm (\Mbfm \vbfm)}_{\eta, q, D; \Ycal}
    \leq \sum_{\betabfm \leq \alphabfm}
    \binom{\alphabfm}{\betabfm}
    \normt[\big]{\pdif_\ybfm^\betabfm \Mbfm}_{\eta, p, D; \Bcal(\Xcal; \Ycal)}
    \normt[\big]{\pdif_\ybfm^{\alphabfm-\betabfm} \vbfm}_{\eta, p', D; \Xcal} .
  \end{equation*}
  for all $\alphabfm \leq \nubfm$.
\end{lemma}
\begin{proof}
  The Leibniz rule and the application of Lemma~\ref{lemma:prodrule} yield:
  \begin{align*}
    &\normt[\big]{\pdif_\ybfm^\alphabfm (\Mbfm \vbfm)}_{\eta, q, D; \Ycal}
    \leq \sum_{\betabfm \leq \alphabfm}
    \binom{\alphabfm}{\betabfm}
    \normt[\big]{\pdif_\ybfm^\betabfm \Mbfm \pdif_\ybfm^{\alphabfm-\betabfm} \vbfm}_{\eta, p, D; \Ycal} \\
    &\quad = \sum_{\betabfm \leq \alphabfm}
    \binom{\alphabfm}{\betabfm}
    \esssup_{\ybfm \in \square}
    \norm[\big]{\pdif_\ybfm^\betabfm \Mbfm[\ybfm] \pdif_\ybfm^{\alphabfm-\betabfm} \vbfm[\ybfm]}_{\eta, p, D; \Ycal} \\
    &\quad \leq \sum_{\betabfm \leq \alphabfm}
    \binom{\alphabfm}{\betabfm}
    \esssup_{\ybfm \in \square}
    \norm[\big]{\pdif_\ybfm^\betabfm \Mbfm[\ybfm]}_{\eta, p, D; \Bcal(\Xcal; \Ycal)}
    \norm[\big]{\pdif_\ybfm^{\alphabfm-\betabfm} \vbfm[\ybfm]}_{\eta, p', D; \Xcal} \\
    &\quad \leq \sum_{\betabfm \leq \alphabfm}
    \binom{\alphabfm}{\betabfm}
    \normt[\big]{\pdif_\ybfm^\betabfm \Mbfm}_{\eta, p, D; \Bcal(\Xcal; \Ycal)}
    \normt[\big]{\pdif_\ybfm^{\alphabfm-\betabfm} \vbfm}_{\eta, p', D; \Xcal} . \qedhere
  \end{align*}
\end{proof}

Again, similar as before by induction, we arrive at:
\begin{corollary}\label{corollary:prodrules11}
  Let $\eta, \in \Nbbb$, $1 \leq p, p_1, \ldots, p_r \leq \infty$, $\nubfm \in \Nbbb^M$,
  $\Xcal_1, \ldots, \Xcal_r$ and $ \Ycal$ be Banach spaces
  and
  \begin{equation*}
    \pdif_\ybfm^\alphabfm \Mbfm \in L_{\Pbbb_\ybfm}^\infty\groupp[\big]{\square; W^{\eta, p}\groupp[\big]{D; \Bcal(\Xcal_1, \ldots, \Xcal_r; \Ycal)}} , \quad
    \pdif_\ybfm^\alphabfm \vbfm_j \in L_{\Pbbb_\ybfm}^\infty\groupp[\big]{\square; W^{\eta, p_j}(D; \Xcal_j)}
  \end{equation*}
  for all $\alphabfm \leq \nubfm$
  with $q = (p^{-1} + p_1^{-1} + \cdots + p_r^{-1})^{-1} \geq 1$.
  Then, we have
  \begin{equation*}
    \pdif_\ybfm^\alphabfm (\Mbfm \vbfm_1 \cdots \vbfm_r) \in L_{\Pbbb_\ybfm}^\infty\groupp[\big]{\square; W^{\eta, p}(D; \Ycal)}
  \end{equation*}
  with
  \begin{align*}
    &\normt[\big]{\pdif_\ybfm^\alphabfm (\Mbfm \vbfm_1 \cdots \vbfm_r)}_{\eta, q, D; \Ycal} \\
    &\quad \leq \sum_{\betabfm + \betabfm_1 + \cdots + \betabfm_r = \alphabfm}
    \binom{\alphabfm}{\betabfm, \betabfm_1, \ldots \betabfm_r}
    \normt[\big]{\pdif_\ybfm^{\betabfm} \Mbfm}_{\eta, p, D; \Bcal(\Xcal_1, \ldots, \Xcal_r; \Ycal)}
    \prod_{j=1}^r \normt[\big]{\pdif_\ybfm^{\betabfm_j} \vbfm_j}_{\eta, p_j, D; \Xcal_j} .
  \end{align*}
  for all $\alphabfm \leq \nubfm$.
\end{corollary}
Moreover, we also will require the following corollary:
\begin{corollary}\label{corollary:prodrules22b}
  Let $\eta, \in \Nbbb$, $1 \leq p_1, p_2 \leq \infty$, $\nubfm \in \Nbbb^M$,
  $\Xcal_1, \Xcal_2$ and $ \Ycal$ be Banach spaces
  and
  \begin{equation*}
    \Mbfm \in \Bcal(\Xcal_1, \Xcal_2; \Ycal) , \quad
    \pdif_\ybfm^\alphabfm \vbfm_j \in L_{\Pbbb_\ybfm}^\infty\groupp[\big]{\square; W^{\eta, p_j}(D; \Xcal_j)}
  \end{equation*}
  for all $\alphabfm \leq \nubfm$
  with $q = (p_1^{-1} + p_2^{-1})^{-1} \geq 1$.
  Then, we have
  \begin{equation*}
    \pdif_\ybfm^\alphabfm (\Mbfm \vbfm_1 \vbfm_2) \in L_{\Pbbb_\ybfm}^\infty\groupp[\big]{\square; W^{\eta, p}(D; \Ycal)}
  \end{equation*}
  with
  \begin{align*}
    &\normt[\big]{\pdif_\ybfm^\alphabfm (\Mbfm \vbfm_1 \vbfm_2)}_{\eta, q, D; \Ycal} \\
    &\quad \leq \norm[\big]{\Mbfm}_{\Bcal(\Xcal_1, \Xcal_2; \Ycal)}
    \sum_{\betabfm \leq \alphabfm}
    \binom{\alphabfm}{\betabfm}
    \normt[\big]{\pdif_\ybfm^{\betabfm} \vbfm_1}_{\eta, p_1, D; \Xcal_1}
    \normt[\big]{\pdif_\ybfm^{\alphabfm-\betabfm} \vbfm_2}_{\eta, p_2, D; \Xcal_2}.
  \end{align*}
  for all $\alphabfm \leq \nubfm$.
\end{corollary}

Lastly, from Theorem~\ref{theorem:comprules} we can the derive the following:
\begin{theorem}\label{theorem:comprules}
  Let $\eta \in \Nbbb$, $1 \leq p \leq \infty$, $\nubfm \in \Nbbb^M$,
  $\Xcal$ and $\Ycal$ be Banach spaces,
  \begin{equation*}
    \vbfm \colon \square \to  W^{\eta, \infty}(D; \Xcal) ,
  \end{equation*}
  $X \subset \Xcal$ be open with $\Img_\square \Img_D \vbfm \subset X$,
  $\Wbfm \colon X \to \Ycal$ be $\eta + \norms{\nubfm}$-times Fr\'echet differentiable
  and, for $\alphabfm \leq \nubfm$ and $0 \leq t \leq \eta + \norms{\nubfm}$,
  \begin{equation*}
    \pdif_\ybfm^\alphabfm \vbfm \in L_{\Pbbb_\ybfm}^\infty\groupp[\big]{\square; W^{\eta, \infty}(D; \Xcal)} , \quad 
    \quad
    \Dif^t \Wbfm \circ \vbfm \in L_{\Pbbb_\ybfm}^\infty\groupp[\big]{\square; L^{p}\groupp[\big]{D; \Bcal^{t}(\Xcal; \Ycal)}} .
  \end{equation*}
  Then, we have
  \begin{equation*}
    \pdif_\ybfm^\alphabfm (\Wbfm \circ \vbfm) \in L_{\Pbbb_\ybfm}^\infty\groupp[\big]{\square; W^{\eta, p}(D; \Ycal)}
  \end{equation*}
  with
  \begin{equation*}
    \normt{\Wbfm \circ \vbfm}_{\eta, p, D; \Ycal}
    \leq \sum_{r=0}^{\eta}
    \frac{1}{r!} \normt[\big]{\Dif^{r} \Wbfm \circ \vbfm}_{p, D; \Bcal^{r}(\Xcal; \Ycal)}
    \normt{\vbfm}_{\eta, \infty, D; \Xcal}^r
  \end{equation*}
  and, for $\alphabfm \neq \zerobfm$,
  \begin{align*}
    \normt[\big]{\pdif_\ybfm^\alphabfm (\Wbfm \circ \vbfm)}_{\eta, p, D; \Ycal}
    &\leq \alphabfm!
    \sum_{s=1}^{\norms{\alphabfm}} \frac{1}{s!}
    \groupp[\bigg]{
      \sum_{r=0}^{\eta}
      \frac{1}{r!} \normt[\big]{\Dif^{r+s} \Wbfm \circ \vbfm}_{p, D; \Bcal^{r+s}(\Xcal; \Ycal)}
      \normt{\vbfm}_{\eta, \infty, D; \Xcal}^r} \\
    &\qquad \qquad \sum_{P(\alphabfm, s)}
    \prod_{j=1}^s \frac{1}{\betabfm_j!}
    \normt[\big]{\pdif_\ybfm^{\betabfm_j} \vbfm}_{\eta, \infty, D; \Xcal} .
  \end{align*}
\end{theorem}
\begin{proof}
  The application of Lemma~\ref{lemma:comprule} leads to
  \begin{align*}
    &\normt[\big]{\Dif^{s} \Wbfm \circ \vbfm}_{\eta, p, D; \Bcal^{s}(\Xcal; \Ycal)} \\
    &\quad = \esssup_{\ybfm \in \square}
    \norm[\big]{\Dif^{s} \Wbfm \circ \vbfm[\ybfm]}_{\eta, p, D; \Bcal^{s}(\Xcal; \Ycal)} \\
    &\quad \leq \esssup_{\ybfm \in \square}
    \sum_{r=0}^{\eta} \frac{1}{r!}
    \norm[\Big]{\Dif^r (\Dif^{s} \Wbfm) \circ \vbfm[\ybfm]}_{p, D; \Bcal^{r}(\Xcal; \Bcal^{s}(\Xcal; \Ycal))}
    \norm[\big]{\vbfm[\ybfm]}_{\eta, \infty, D; \Xcal}^r \\
    &\quad \leq  \sum_{r=0}^{\eta} \frac{1}{r!}
    \normt[\big]{\Dif^{r+s} \Wbfm \circ \vbfm}_{p, D; \Bcal^{r+s}(\Xcal; \Ycal)}
    \normt{\vbfm}_{\eta, \infty, D; \Xcal}^r ,
  \end{align*}
  from which, with $s=0$, the assertion follows directly for $\alphabfm = \zerobfm$.
  
  For $\alphabfm \neq \zerobfm$,
  we remark that the Fa\`a di Bruno formula and Corollary~\ref{corollary:prodrules11} yield
  \begin{align*}
    \normt[\big]{\pdif_\ybfm^\alphabfm (\Wbfm \circ \vbfm)}_{\eta, p, D; \Ycal}
    &\leq \alphabfm!
    \sum_{s=1}^{\norms{\alphabfm}} \frac{1}{s!}
    \sum_{P(\alphabfm, s)}
    \normt[\bigg]{(\Dif^{s} \Wbfm \circ \vbfm)
      \pdif_\ybfm^{\betabfm_1} \vbfm
      \cdots \pdif_\ybfm^{\betabfm_s} \vbfm
      \prod_{j=1}^s \frac{1}{\betabfm_j!}}_{\eta, p, D; \Ycal} \\
    &\leq \alphabfm!
    \sum_{s=1}^{\norms{\alphabfm}} \frac{1}{s!}
    \sum_{P(\alphabfm, s)}
    \normt[\big]{\Dif^{s} \Wbfm \circ \vbfm}_{\eta, p, D; \Bcal^{s}(\Xcal; \Ycal)}
    \prod_{j=1}^s \frac{1}{\betabfm_j!}
    \normt[\big]{\pdif_\ybfm^{\betabfm_j} \vbfm}_{\eta, \infty, D; \Xcal} \\
    &\leq \alphabfm!
    \sum_{s=1}^{\norms{\alphabfm}} \frac{1}{s!}
    \normt[\big]{\Dif^{s} \Wbfm \circ \vbfm}_{\eta, p, D; \Bcal^{s}(\Xcal; \Ycal)}
    \sum_{P(\alphabfm, s)}
    \prod_{j=1}^s \frac{1}{\betabfm_j!}
    \normt[\big]{\pdif_\ybfm^{\betabfm_j} \vbfm}_{\eta, \infty, D; \Xcal} . \qedhere
  \end{align*}
\end{proof}
\subsection{Parametric regularity}
We first consider the two terms
\begin{equation*}
  \Bbfm[\ybfm](\xbfm) \isdef \Vbfm[\ybfm](\xbfm) \Vbfm^\trans[\ybfm](\xbfm)
  \quad\text{and}\quad
  C[\ybfm](\xbfm) \isdef \Vbfm^\trans[\ybfm](\xbfm) \Vbfm[\ybfm](\xbfm)
\end{equation*}
which appear in the diffusion coefficient.
For these we have:
\begin{lemma}\label{lemma:BCybounds}
  We have for all $\alphabfm \in \Nbbb^M$ that
  \begin{equation*}
    \normt[\big]{\pdif_\ybfm^\alphabfm \Bbfm}_{\kappa, \infty, D}
    \leq k_\Bbfm \gammabfm_{\kappa}^\alphabfm
    \quad\text{and}\quad
    \normt[\big]{\pdif_\ybfm^\alphabfm C}_{\kappa, \infty, D}
    \leq k_C \gammabfm_{\kappa}^\alphabfm ,
  \end{equation*}
  with $k_\Bbfm \isdef k_C \isdef c_{\gammabfm_\kappa}^2$.
  Furthermore,
  we know that $\Img_\square \Img_D C \subset {[\underline{a}^2, \overline{a}^2]}$.
\end{lemma}
\begin{proof}
  More verbosely, $\Bbfm$ is given by
  \begin{equation*}
    \Bbfm[\ybfm](\xbfm)
    = \groupp[\bigg]{\psibfm_0(\xbfm)  + \sum_{k=1}^{M} \sigma_k \psibfm_k(\xbfm) y_k}
    \groupp[\bigg]{\psibfm_0(\xbfm)  + \sum_{k=1}^{M} \sigma_k \psibfm_k(\xbfm) y_k}^\trans ,
  \end{equation*}
  from which we can derive the first order derivatives, yielding
  \begin{equation}
    \label{eq:dyiB}
    \begin{aligned}
      \pdif_{y_i} \Bbfm[\ybfm](\xbfm)
      &= \sigma_i \psibfm_i(\xbfm)
      \groupp[\bigg]{\psibfm_0(\xbfm)  + \sum_{k=1}^{M} \sigma_k \psibfm_k(\xbfm) y_k}^\trans \\
      &\quad {} + \groupp[\bigg]{\psibfm_0(\xbfm)  + \sum_{k=1}^{M} \sigma_k \psibfm_k(\xbfm) y_k}
      \sigma_i \psibfm_i^\trans(\xbfm) ,
    \end{aligned}
  \end{equation}
  and from those also the second order derivatives. They are given by
  \begin{equation}
    \label{eq:dyjdyiB}
    \pdif_{y_j} \pdif_{y_i} \Bbfm[\ybfm](\xbfm)
    = \sigma_i \psibfm_i(\xbfm) \sigma_j \psibfm_j^\trans(\xbfm)
    + \sigma_j \psibfm_j(\xbfm) \sigma_i \psibfm_i^\trans(\xbfm) .
  \end{equation}
  Since the second order derivatives with respect to $\ybfm$ are constant,
  all higher order derivatives with respect to $\ybfm$ vanish.
  
  We obviously have
  \begin{equation*}
    \normt{\Bbfm}_{\kappa, \infty, D} \leq c_{\gammabfm_\kappa}^2 .
  \end{equation*}
  From \eqref{eq:dyiB} we can now derive the bound
  \begin{equation*}
    \normt[\big]{\pdif_{y_i} \Bbfm}_{\kappa, \infty, D}
    \leq 2 \normt[\big]{\sigma_i \psibfm_i}_{\kappa, \infty, D} \normt[\bigg]{\psibfm_0  + \sum_{k=1}^{M} \sigma_k \psibfm_k y_k}_{\kappa, \infty, D}
    \leq 2 \gamma_{\kappa, i} c_{\gammabfm_\kappa}
  \end{equation*}
  and \eqref{eq:dyjdyiB} leads us to
  \begin{equation*}
    \normt[\big]{\pdif_{y_j} \pdif_{y_i} \Bbfm}_{\kappa, \infty, D}
    \leq 2 \normt[\big]{\sigma_i \psibfm_i}_{\kappa, \infty, D} \normt[\big]{\sigma_j \psibfm_j}_{\kappa, \infty, D}
    \leq 2 \gamma_{\kappa, i} \gamma_{\kappa, j} .
  \end{equation*}
  Therefore, we have
  \begin{equation*}
    \normt[\big]{\pdif_\ybfm^\alphabfm \Bbfm}_{\kappa, \infty, D}
    \leq \begin{cases}
      c_{\gammabfm_\kappa}^2 \gammabfm_\kappa^\alphabfm , & \text{if $\norms{\alphabfm} = 0$,} \\
      2 c_{\gammabfm_\kappa} \gammabfm_\kappa^\alphabfm , & \text{if $\norms{\alphabfm} = 1$,} \\
      2 \gammabfm_\kappa^\alphabfm , & \text{if $\norms{\alphabfm} = 2$,} \\
      0 , & \text{if $\norms{\alphabfm} > 2$,}
    \end{cases}
  \end{equation*}
  and are finished since $c_{\gammabfm_\kappa} \geq 2$.

  Starting from the equation
  \begin{equation*}
    C[\ybfm](\xbfm)
    = \groupp[\bigg]{\psibfm_0(\xbfm)  + \sum_{k=1}^{M} \sigma_k \psibfm_k(\xbfm) y_k}^\trans
    \groupp[\bigg]{\psibfm_0(\xbfm)  + \sum_{k=1}^{M} \sigma_k \psibfm_k(\xbfm) y_k} ,
  \end{equation*}
  the bound for $C$ clearly follows analogously as for $\Bbfm$.
  Lastly, we note that \eqref{eq:Vellipticity} directly implies
  that $\Img_\square \Img_D C \subset {[\underline{a}^2, \overline{a}^2]}$.
\end{proof}
Next, we consider the terms
\begin{equation*}
  D[\ybfm](\xbfm) \isdef \groupp[\big]{C[\ybfm](\xbfm)}^{-1}
  \quad\text{and}\quad
  E[\ybfm](\xbfm) \isdef \sqrt{C[\ybfm](\xbfm)} ,
\end{equation*}
for which we have:
\begin{lemma}\label{lemma:DEybounds}
  We know for all $\alphabfm \in \Nbbb^M$ that
  \begin{equation*}
    \normt[\big]{\pdif_\ybfm^\alphabfm D}_{\kappa, \infty, D}
    \leq \norms{\alphabfm}! k_D c_D^{\norms{\alphabfm}} \gammabfm_\kappa^{\alphabfm}
    \quad\text{and}\quad
    \normt[\big]{\pdif_\ybfm^\alphabfm E}_{\kappa, \infty, D}
    \leq \norms{\alphabfm}! k_E c_E^{\norms{\alphabfm}} \gammabfm_\kappa^{\alphabfm} ,
  \end{equation*}
  with
  \begin{equation*}
    k_D \isdef (\kappa+1) \overline{a} \groupp[\bigg]{\frac{2 k_C}{\underline{a}^2}}^\kappa ,
    \quad
    k_E \isdef (\kappa+1) \frac{1}{\underline{a}^2} \groupp[\bigg]{\frac{2 k_C}{\underline{a}^2}}^\kappa
    \quad\text{and}\quad
    c_D \isdef c_E \isdef \frac{2 k_C}{\underline{a}^2 \log 2} .
  \end{equation*}
\end{lemma}
\begin{proof}
  Because $D = v \circ C$ with $v(x) = x^{-1}$ and $E = w \circ C$ with $w(x) = \sqrt{x}$ are
  composite functions,
  we can employ Theorem~\ref{theorem:comprules}
  to bound their derivatives.
  For this,
  we remark that the $t$-th derivative of $v$ and $w$ are given by
  \begin{equation*}
    \Dif_x^t v(x) h_1 \cdots h_t
    = (-1)^t t! x^{-1-r} \prod_{k=1}^t h_k
    = (-1)^t t! v(x) v(x)^t \prod_{k=1}^t h_k 
  \end{equation*}
  and
  \begin{equation*}
    \Dif_x^t w(x) h_1 \cdots h_t
    = c_t x^{\frac{1}{2}-t} \prod_{k=1}^t h_k
    = c_t w(x) v(x)^t \prod_{k=1}^t h_k ,
  \end{equation*}
  where $c_t \isdef \prod_{i=0}^{t-1} \groupp[\big]{\frac{1}{2} - i}$.
  For $x \in {[\underline{a}^2, \overline{a}^2]}$ we therefore have
  \begin{equation*}
    \norm[\big]{\Dif_x^t v(x)}_{\Bcal^t(\Rbbb; \Rbbb)}
    \leq t! \frac{1}{\underline{a}^2} \groupp[\bigg]{\frac{1}{\underline{a}^2}}^{t}
    \quad\text{and}\quad
    \norm[\big]{\Dif_x^t w(x)}_{\Bcal^t(\Rbbb; \Rbbb)}
    \leq t! \overline{a} \groupp[\bigg]{\frac{1}{\underline{a}^2}}^{t}
  \end{equation*}
  which implies
  \begin{equation*}
    \normt[\big]{\Dif_x^{t} v \circ C}_{\infty, D; \Bcal^{t}(\Rbbb; \Rbbb)}
    \leq t! \frac{1}{\underline{a}^2} \groupp[\bigg]{\frac{1}{\underline{a}^2}}^{t}
    \quad\text{and}\quad
    \normt[\big]{\Dif_x^{t} w \circ C}_{\infty, D; \Bcal^{t}(\Rbbb; \Rbbb)}
    \leq t! \overline{a} \groupp[\bigg]{\frac{1}{\underline{a}^2}}^{t}
  \end{equation*}
  using Lemma~\ref{lemma:BCybounds} --- 
  with which we then furthermore arrive at the estimates
  \begin{align*}
    \normt{v \circ C}_{\kappa, \infty, D}
    &\leq \sum_{r=0}^{\eta} \frac{1}{r!}
    \normt[\big]{\Dif^{r} v \circ C}_{\infty, D; \Bcal^{r}(\Rbbb; \Rbbb)}
    \normt{C}_{\eta, \infty, D}^r \\
    &\leq \sum_{r=0}^{\kappa} \frac{1}{r!}
    r! \frac{1}{\underline{a}^2} \groupp[\bigg]{\frac{1}{\underline{a}^2}}^{r} k_C^r \\
    &\leq (\kappa+1)
    \frac{1}{\underline{a}^2} \groupp[\bigg]{\frac{1}{\underline{a}^2}}^{\kappa} k_C^\kappa
  \end{align*}
  and, analogously,
  \begin{equation*}
    \normt{w \circ C}_{\kappa, \infty, D}
    \leq (\kappa+1)
    \overline{a} \groupp[\bigg]{\frac{1}{\underline{a}^2}}^{\kappa} k_C^\kappa ,
  \end{equation*}
  as well as, for $\alphabfm \neq \zerobfm$,
  \begin{align*}
    \normt[\big]{\pdif_\ybfm^\alphabfm (v \circ C)}_{\kappa, \infty, D}
    &\leq \alphabfm!
    \sum_{s=1}^{\norms{\alphabfm}} \frac{1}{s!}
    \groupp[\bigg]{
      \sum_{r=0}^{\kappa} \frac{1}{r!}
      \normt[\big]{\Dif^{r+s} v \circ C}_{\infty, D; \Bcal^{r+s}(\Rbbb; \Rbbb)}
      \normt{C}_{\kappa, \infty, D}^r} \\
    &\qquad \qquad \sum_{P(\alphabfm, s)}
    \prod_{j=1}^s \frac{1}{\betabfm_j!}
    \normt[\big]{\pdif_\ybfm^{\betabfm_j} C}_{\kappa, \infty, D} \\
    &\leq \alphabfm!
    \sum_{s=1}^{\norms{\alphabfm}} \frac{1}{s!}
    \groupp[\bigg]{
      \sum_{r=0}^{\kappa} \frac{1}{r!}
      (r+s)! \frac{1}{\underline{a}^2} \groupp[\bigg]{\frac{1}{\underline{a}^2}}^{r+s} k_C^r}
    \sum_{P(\alphabfm, s)}
    \prod_{j=1}^s \frac{1}{\betabfm_j!}
    k_C \gammabfm_{\kappa}^{\betabfm_j} \\
    &\leq \alphabfm!
    (\kappa+1) 2^\kappa
    \frac{1}{\underline{a}^2} \groupp[\bigg]{\frac{1}{\underline{a}^2}}^{\kappa}
    k_C^\kappa
    2^{\norms{\alphabfm}} \groupp[\bigg]{\frac{1}{\underline{a}^2}}^{\norms{\alphabfm}}
    \groupp[\bigg]{
      \sum_{s=1}^{\norms{\alphabfm}} \sum_{P(\alphabfm, s)}
      \prod_{j=1}^s \frac{1}{\betabfm_j!}}
    k_C^{\norms{\alphabfm}} \gammabfm_{\kappa}^{\alphabfm} \\
    &\leq \norms{\alphabfm}!
    (\kappa+1) 2^\kappa
    \frac{1}{\underline{a}^2} \groupp[\bigg]{\frac{1}{\underline{a}^2}}^{\kappa} 
    k_C^\kappa
    2^{\norms{\alphabfm}} \groupp[\bigg]{\frac{1}{\underline{a}^2}}^{\norms{\alphabfm}}
    k_C^{\norms{\alphabfm}} \groupp[\bigg]{\frac{1}{\log 2}}^{\norms{\alphabfm}}
    \gammabfm_{\kappa}^{\alphabfm} 
  \end{align*}
  and, again analogously,
  \begin{equation*}
    \normt[\big]{\pdif_\ybfm^\alphabfm (w \circ C)}_{\kappa, \infty, D}
    \leq \norms{\alphabfm}!
    (\kappa+1) 2^\kappa
    \overline{a} \groupp[\bigg]{\frac{1}{\underline{a}^2}}^{\kappa} 
    k_C^\kappa
    2^{\norms{\alphabfm}} \groupp[\bigg]{\frac{1}{\underline{a}^2}}^{\norms{\alphabfm}}
    k_C^{\norms{\alphabfm}} \groupp[\bigg]{\frac{1}{\log 2}}^{\norms{\alphabfm}}
    \gammabfm_{\kappa}^{\alphabfm} .
  \end{equation*}
  Combining these estimates finally yields the assertions.
\end{proof}
We can now consider the two terms
\begin{equation*}
  \Fbfm[\ybfm](\xbfm) \isdef D[\ybfm](\xbfm) \Bbfm[\ybfm](\xbfm)
  \quad\text{and}\quad
  \Gbfm[\ybfm](\xbfm) \isdef E[\ybfm](\xbfm) \Fbfm[\ybfm](\xbfm) .
\end{equation*}
For these we have:
\begin{lemma}\label{lemma:FGybounds}
  We have for all $\alphabfm \in \Nbbb^M$ that
  \begin{equation*}
    \normt[\big]{\pdif_\ybfm^\alphabfm \Fbfm}_{\kappa, \infty, D}
    \leq \norms{\alphabfm}! k_\Fbfm c_\Fbfm^{\norms{\alphabfm}} \gammabfm_{\kappa}^\alphabfm
    \quad\text{and}\quad
    \normt[\big]{\pdif_\ybfm^\alphabfm \Gbfm}_{\kappa, \infty, D}
    \leq \norms{\alphabfm}! k_\Gbfm c_\Gbfm^{\norms{\alphabfm}} \gammabfm_{\kappa}^\alphabfm ,
  \end{equation*}
  with
  \begin{equation*}
    k_\Fbfm \isdef 3 k_D k_\Bbfm ,
    \quad
    c_\Fbfm \isdef c_D ,
    \quad
    k_\Gbfm \isdef k_E k_\Fbfm
    \quad\text{and}\quad
    c_\Gbfm \isdef 2 \max\groupb{c_E, c_\Fbfm} .
  \end{equation*}
\end{lemma}
\begin{proof}
  We use Corollary~\ref{corollary:prodrules22b} with the bounds
  from Lemma~\ref{lemma:BCybounds} and Lemma~\ref{lemma:DEybounds}
  to arrive at
  \begin{align*}
    \normt[\big]{\pdif_\ybfm^\alphabfm (D \Bbfm)}_{\kappa, \infty, D}
    &\leq \sum_{\betabfm \leq \alphabfm} \binom{\alphabfm}{\betabfm} 
    \normt[\big]{\pdif_\ybfm^\betabfm D}_{\kappa, \infty, D}
    \normt[\big]{\pdif_\ybfm^{\alphabfm-\betabfm} \Bbfm}_{\kappa, \infty, D} \\
    &\leq \sum_{\betabfm \leq \alphabfm} \binom{\alphabfm}{\betabfm}
    \norms{\betabfm}! k_D c_D^{\norms{\betabfm}} \gammabfm_\kappa^{\betabfm}
    k_\Bbfm \gammabfm_\kappa^{\alphabfm-\betabfm} \\
    &\leq k_D k_\Bbfm c_D^{\norms{\alphabfm}}
    \gammabfm_\kappa^{\alphabfm}
    \sum_{\betabfm \leq \alphabfm} \binom{\alphabfm}{\betabfm} \norms{\betabfm}!
  \end{align*}
  and
  \begin{align*}
    \normt[\big]{\pdif_\ybfm^\alphabfm (E \Fbfm)}_{\kappa, \infty, D}
    &\leq \sum_{\betabfm \leq \alphabfm} \binom{\alphabfm}{\betabfm} 
    \normt[\big]{\pdif_\ybfm^\betabfm E}_{\kappa, \infty, D}
    \normt[\big]{\pdif_\ybfm^{\alphabfm-\betabfm} \Fbfm}_{\kappa, \infty, D} \\
    &\leq \sum_{\betabfm \leq \alphabfm} \binom{\alphabfm}{\betabfm}
    \norms{\betabfm}! k_E c_E^{\norms{\betabfm}} \gammabfm_\kappa^{\betabfm}
    \norms{\alphabfm-\betabfm}! k_\Fbfm c_\Fbfm^{\norms{\alphabfm-\betabfm}} \gammabfm_\kappa^{\alphabfm-\betabfm} \\
    &\leq k_E k_\Fbfm \max\groupb{c_E, c_\Fbfm}^{\norms{\alphabfm}}
    \gammabfm_\kappa^{\alphabfm}
    \sum_{\betabfm \leq \alphabfm} \binom{\alphabfm}{\betabfm}
    \norms{\betabfm}! \norms{\alphabfm-\betabfm}!
  \end{align*}
  
  Lastly, the combinatorial identity
  \begin{equation}
    \label{eq:combinatorialidentity}
    \sum_{\substack{\betabfm \leq \alphabfm\\\norms{\betabfm}=j}} \binom{\alphabfm}{\betabfm}
    = \binom{\norms{\alphabfm}}{j}
  \end{equation}
  yields the bounds
  \begin{equation*}
    \sum_{\betabfm \leq \alphabfm} \binom{\alphabfm}{\betabfm} \norms{\betabfm}!
    = \sum_{j=0}^{\norms{\alphabfm}} j!
    \sum_{\substack{\betabfm \leq \alphabfm\\\norms{\betabfm}=j}} \binom{\alphabfm}{\betabfm}
    = \sum_{j=0}^{\norms{\alphabfm}} j! \binom{\norms{\alphabfm}}{j}
    = \norms{\alphabfm}! \sum_{k=0}^{\norms{\alphabfm}} \frac{1}{k!}
    \leq 3 \norms{\alphabfm}! .
  \end{equation*}
  and
  \begin{equation*}
    \sum_{\betabfm \leq \alphabfm} \binom{\alphabfm}{\betabfm} \norms{\betabfm}! \norms{\alphabfm-\betabfm}!
    = \sum_{j=0}^{\norms{\alphabfm}} j! (\norms{\alphabfm}-j)!
    \sum_{\substack{\betabfm \leq \alphabfm\\\norms{\betabfm}=j}} \binom{\alphabfm}{\betabfm}
    = \sum_{j=0}^{\norms{\alphabfm}} j! (\norms{\alphabfm}-j)! \binom{\norms{\alphabfm}}{j}
    = \norms{\alphabfm}! \sum_{k=0}^{\norms{\alphabfm}} 1
    \leq 2^{\norms{\alphabfm}} \norms{\alphabfm}!
  \end{equation*}
  with which the assertion follows.
\end{proof}
We now can consider the complete diffusion coefficient,
yielding:
\begin{theorem}\label{theorem:Aybounds}
  The derivatives of the diffusion matrix $\Abfm$, defined in \eqref{eq:AV}, satisfy
  \begin{equation*}
    \normt[\big]{\pdif_\ybfm^\alphabfm \Abfm}_{\kappa, \infty, D}
    \leq \norms{\alphabfm}! k_\Abfm c_\Abfm^{\norms{\alphabfm}} \gammabfm_{\kappa}^\alphabfm
  \end{equation*}
  for all $\alphabfm \in \Nbbb^M$ with
  \begin{equation*}
    k_\Abfm \isdef k_\Gbfm + \overline{a} k_\Fbfm + \overline{a}
    \quad\text{and}\quad
    c_\Abfm \isdef \max\groupb{c_\Gbfm, c_\Fbfm} .
  \end{equation*}
\end{theorem}
\begin{proof}
  We can state $\Abfm$ as
  \begin{equation*}
    \Abfm[\ybfm](\xbfm) = a \Ibfm + \Gbfm[\ybfm](\xbfm) - a \Fbfm[\ybfm](\xbfm) ,
  \end{equation*}
  which,
  by taking the norm of the derivative and using Lemma~\ref{lemma:FGybounds},
  yields
  \begin{align*}
    \normt[\big]{\pdif_\ybfm^\alphabfm \Abfm}_{\kappa, \infty, D}
    &\leq \overline{a} \normt[\big]{\pdif_\ybfm^\alphabfm \Ibfm}_{\kappa, \infty, D}
    + \normt[\big]{\pdif_\ybfm^\alphabfm \Gbfm}_{\kappa, \infty, D}
    + \overline{a} \normt[\big]{\pdif_\ybfm^\alphabfm \Fbfm}_{\kappa, \infty, D} \\
    &\leq\overline{a}
    + \norms{\alphabfm}! k_\Gbfm c_\Gbfm^{\norms{\alphabfm}} \gammabfm_{\kappa}^\alphabfm
    + \overline{a} \norms{\alphabfm}! k_\Fbfm c_\Fbfm^{\norms{\alphabfm}} \gammabfm_{\kappa}^\alphabfm \\
    &\leq \norms{\alphabfm}! (k_\Gbfm + \overline{a} k_\Fbfm + \overline{a})
    \max\groupb{c_\Gbfm, c_\Fbfm}^{\norms{\alphabfm}} \gammabfm_{\kappa}^\alphabfm . \qedhere
  \end{align*}
\end{proof}

We now define the modified sequence $\mubfm_\kappa = (\mu_{\kappa, k})_{k \in \Nbbb}$ as
\begin{equation*}
  \mu_{\kappa, k}
  \isdef c_\Abfm \gamma_{\kappa, k}
  = \frac{4 c_{\gammabfm_\kappa}^2}{\underline{a}^2 \log 2} \gamma_{\kappa, k}
\end{equation*}
and also
\begin{equation*}
  k_{\kappa, \Abfm}
  \isdef k_\Abfm
  = \groups[\bigg]{(\kappa+1) \frac{1}{\underline{a}^2} \groupp[\bigg]{\frac{2 c_{\gammabfm_\kappa}^2}{\underline{a}^2}}^\kappa + \overline{a}} 3 (\kappa+1) \overline{a} \groupp[\bigg]{\frac{2 c_{\gammabfm_\kappa}^2}{\underline{a}^2}}^\kappa c_{\gammabfm_\kappa}^2 + \overline{a} ;
\end{equation*}
thus, we have
\begin{equation*}
  \normt[\big]{\pdif_\ybfm^\alphabfm \Abfm}_{\kappa, \infty, D}
  \leq \norms{\alphabfm}! k_{\kappa, \Abfm} \mubfm_\kappa^{\alphabfm} .
\end{equation*}

Now, Assumption~\ref{assumption:er} directly implies the following result:
\begin{lemma}\label{lemma:eubound}
  The unique solution $u \in L_{\Pbbb_\ybfm}^\infty\groupp[\big]{\square; H_0^1(D)}$
  of \eqref{eq:swpsodp} fulfils
  $u \in L_{\Pbbb_\ybfm}^\infty\groupp[\big]{\square; H^{\kappa+1}(D)}$,
  with
  \begin{equation*}
    \normt{u}_{\kappa+1, 2, D}
    \leq C_{\kappa, er} \norm{f}_{\kappa-1, 2, D} ,
  \end{equation*}
  where
  \begin{equation*}
    C_{\kappa, er} \isdef \max_{0 \leq s \leq \norm{\Abfm}_{L_{\Pbbb_\ybfm}^\infty(\square; \Rcal_\kappa)}} C_{\kappa, er}(D, s) .
  \end{equation*}
\end{lemma}

However,
by also leveraging the higher spatial regularity in the Karhunen\hyp{}Lo\`eve
expansion of the random vector\hyp{}valued field,
we can show that the solution $u$ admits analytic regularity 
with respect to the stochastic parameter $\ybfm$ also in the $H^{\kappa+1}(D)$-norm.
This mixed regularity is then the essential ingredient when applying multilevel methods.
\begin{theorem}\label{theorem:euybounds}
  The derivatives in $\ybfm$ of the solution $u$ of \eqref{eq:swpsodp} satisfy
  \begin{equation*}
    \normt[\big]{\pdif_\ybfm^\alphabfm u}_{\kappa+1, 2, D} 
    \leq \norms{\alphabfm}! \mubfm_{\kappa}^\alphabfm
    \groupp[\bigg]{\max\groupb[\Big]{2, 2 C_{\kappa, er} \kappa^2 d^2 k_{\kappa, \Abfm},
      C_{\kappa, er} \norm{f}_{\kappa-1, 2, D}}}^{\norms{\alphabfm}+1} .
  \end{equation*}
\end{theorem}
\begin{proof}
  By differentiation of the variational formulation \eqref{eq:swpsodp} with respect to $\ybfm$
  we arrive, for arbitrary $v \in H_0^1(D)$, at
  \begin{equation*}
    \groupp[\Big]{\pdif_\ybfm^\alphabfm \groupp[\big]{\Abfm \Grad_\xbfm u},
    \Grad_\xbfm v}_{L^2(D; \Rbbb^d)} = 0 .
  \end{equation*}
  Applying the Leibniz rule on the left\hyp{}hand side yields
  \begin{equation*}
    \groupp[\bigg]{\sum_{\betabfm \leq \alphabfm} \binom{\alphabfm}{\betabfm}
    \pdif_\ybfm^{\alphabfm-\betabfm} \Abfm
    \pdif_\ybfm^{\betabfm} \Grad_\xbfm u ,
    \Grad_\xbfm v}_{L^2(D; \Rbbb^d)} = 0 .
  \end{equation*}
  Then, by rearranging and using the linearity of the gradient, we find
  \begin{equation*}
    \groupp[\Big]{
      \Abfm
      \Grad_\xbfm \pdif_\ybfm^{\alphabfm} u ,
      \Grad_\xbfm v}_{L^2(D; \Rbbb^d)}
    = - \groupp[\bigg]{
      \sum_{\betabfm < \alphabfm} \binom{\alphabfm}{\betabfm}
      \pdif_\ybfm^{\alphabfm-\betabfm} \Abfm
      \Grad_\xbfm \pdif_\ybfm^{\betabfm} u ,
      \Grad_\xbfm v}_{L^2(D; \Rbbb^d)} .
  \end{equation*}
  Using Green's identity, we can then write
  \begin{equation*}
    \groupp[\Big]{
      \Abfm
      \Grad_\xbfm \pdif_\ybfm^{\alphabfm} u ,
      \Grad_\xbfm v}_{L^2(D; \Rbbb^d)}
    = \groupp[\bigg]{
      \sum_{\betabfm < \alphabfm} \binom{\alphabfm}{\betabfm}
      \Div_\xbfm \groupp[\Big]{
        \pdif_\ybfm^{\alphabfm-\betabfm} \Abfm
        \Grad_\xbfm \pdif_\ybfm^{\betabfm} u} ,
      v}_{L^2(D; \Rbbb)} .
  \end{equation*}
  Thus, we arrive at
  \begin{align*}
    \normt[\big]{\pdif_\ybfm^{\alphabfm} u}_{\kappa+1, 2, D}
    &\leq C_{\kappa, er}
    \sum_{\betabfm < \alphabfm} \binom{\alphabfm}{\betabfm}
    \normt[\Big]{\Div_\xbfm
      \groupp[\Big]{
        \pdif_\ybfm^{\alphabfm-\betabfm} \Abfm
        \Grad_\xbfm \pdif_\ybfm^{\betabfm} u}}_{\kappa-1, 2, D} \\
    &\leq C_{\kappa, er}
    \sum_{\betabfm < \alphabfm} \binom{\alphabfm}{\betabfm}
    \kappa d \normt[\Big]{
      \pdif_\ybfm^{\alphabfm-\betabfm} \Abfm
      \Grad_\xbfm \pdif_\ybfm^{\betabfm} u}_{\kappa, 2, D} \\
    &\leq C_{\kappa, er}
    \sum_{\betabfm < \alphabfm} \binom{\alphabfm}{\betabfm}
    \kappa d
    \normt[\big]{\pdif_\ybfm^{\alphabfm-\betabfm} \Abfm}_{\kappa, \infty, D}
    \normt[\big]{\Grad_\xbfm \pdif_\ybfm^{\betabfm} u}_{\kappa, 2, D} \\
    &\leq C_{\kappa, er} \kappa^2 d^2 k_{\kappa, \Abfm}
    \sum_{\betabfm < \alphabfm} \binom{\alphabfm}{\betabfm}
    \norms{\alphabfm-\betabfm}! \mubfm_{\kappa}^{\alphabfm-\betabfm}
    \normt[\big]{\pdif_\ybfm^{\betabfm} u}_{\kappa+1, 2, D},
  \end{align*}
  from which we derive
  \begin{equation*}
    \normt[\big]{\pdif_\ybfm^\alphabfm u}_{\kappa+1, 2, D}
    \leq \frac{c}{2}
    \sum_{\betabfm < \alphabfm} \binom{\alphabfm}{\betabfm}
    \norms{\alphabfm-\betabfm}! \mubfm_{\kappa}^{\alphabfm-\betabfm}
    \normt[\big]{\pdif_\ybfm^{\betabfm} u}_{\kappa+1, 2, D} ,
  \end{equation*}
  where
  \begin{equation*}
    c \isdef
    \max\groupb[\Big]{2, 2 C_{\kappa, er} \kappa^2 d^2 k_{\kappa, \Abfm},
      C_{\kappa, er} \norm{f}_{\kappa-1, 2, D}} .
  \end{equation*}

  We note that, by definition of $c$, we have $c \geq 2$
  and furthermore, because of Lemma~\ref{lemma:eubound}, we also have that
  $\normt[\big]{u}_{H^1(D)} \leq c$,
  which means that the assertion is true for $\norms{\alphabfm} = 0$.
  Thus, we can use an induction over $\norms{\alphabfm}$ to prove the hypothesis
  \begin{equation*}
    \normt[\big]{\pdif_\ybfm^\alphabfm u}_{\kappa+1, 2, D} 
    \leq \norms{\alphabfm}! \mubfm_{\kappa}^\alphabfm c^{\norms{\alphabfm}+1}
  \end{equation*}
  for $\norms{\alphabfm} > 0$.
  
  Let the assertions hold for all $\alphabfm$,
  which satisfy $\norms{\alphabfm} \leq n-1$ for some $n \geq 1$.
  Then, we know for all $\alphabfm$ with $\norms{\alphabfm} = n$ that
  \begin{align*}
    \normt[\big]{\pdif_\ybfm^\alphabfm u}_{\kappa+1, 2, D}
    &\leq \frac{c}{2}
    \sum_{\betabfm < \alphabfm} \binom{\alphabfm}{\betabfm}
    \norms{\alphabfm-\betabfm}! \mubfm_{\kappa}^{\alphabfm-\betabfm}
    \normt[\big]{\pdif_\ybfm^{\betabfm} u}_{\kappa+1, 2, D} \\
    &\leq \frac{c}{2} \mubfm_{\kappa}^\alphabfm
    \sum_{\betabfm < \alphabfm} \binom{\alphabfm}{\betabfm}
    \norms{\alphabfm-\betabfm}! \norms{\betabfm}!
    c^{\norms{\betabfm}+1} \\
    &= \frac{c}{2} \mubfm_{\kappa}^\alphabfm
    \sum_{j=0}^{n-1} \sum_{\substack{\betabfm < \alphabfm\\\norms{\betabfm}=j}} \binom{\alphabfm}{\betabfm}
    \norms{\alphabfm-\betabfm}! \norms{\betabfm}!
    c^{\norms{\betabfm}+1} .
  \end{align*}
  Making use of the combinatorial identity \eqref{eq:combinatorialidentity} yields
  \begin{align*}
    \normt[\big]{\pdif_\ybfm^\alphabfm u}_{\kappa+1, 2, D}
    &\leq \frac{c}{2} \mubfm_{\kappa}^\alphabfm
    \sum_{j=0}^{n-1} \binom{\norms{\alphabfm}}{j}
    (\norms{\alphabfm}-j)! j! c^{j+1} 
    = \frac{c}{2} \norms{\alphabfm}! \mubfm_{\kappa}^\alphabfm
    c \sum_{j=0}^{n-1} c^j \\
    &\leq \frac{c}{2(c-1)} \norms{\alphabfm}! \mubfm_{\kappa}^\alphabfm
    c^{\norms{\alphabfm}+1} .
  \end{align*}
  Now, since $c \geq 2$, we have $c \leq 2(c-1)$ and hence also
  \begin{equation*}
    \normt[\big]{\pdif_\ybfm^\alphabfm u}_{\kappa+1, 2, D}
    \leq \norms{\alphabfm}! \mubfm_{\kappa}^\alphabfm c^{\norms{\alphabfm}+1} .
  \end{equation*}
  This completes the proof.
\end{proof}
\subsection{Numerical quadrature in the parameter}
Coming from the solution of \eqref{eq:swpsodp},
that is $u \in L_{\Pbbb_\ybfm}^\infty\groupp[\big]{\square; H^{\kappa+1}(D)}$,
we now wish to know the moments of $u$.
In this section, we will therefore consider the approximation of the mean of $u$.

The mean of $u$ is given by the Bochner integral
\begin{equation*}
  \Mean[u](\xbfm) = \int_\square u[\ybfm](\xbfm) \dif\!\ybfm.
\end{equation*}
Therefore, we may proceed to approximate it
by considering a generic quadrature method $Q_N$; that is
\begin{equation}\label{eq:quad}
  \Mean[u](\xbfm)
  \approx Q_N[u](\xbfm)
  \isdef \sum_{i=1}^N \omega_i^{(N)} u\groupp[\big]{\xbfm, \xibfm_i^{(N)}} ,
\end{equation}
where
\begin{equation*}
  \groupb[\Big]{\groupp[\big]{\omega_i^{(N)}, \xibfm_i^{(N)}}}_{i=1}^{N} \subset \Rbbb \times [0, 1]^M
\end{equation*}
are the weight and evaluation point pairs.
We assume that the quadrature chosen fulfils
\begin{equation}\label{eq:quadr}
  \norm[\big]{\Mean[u] - Q_N[u]}_{H^1(D)} \leq C N^{-r}
\end{equation}
for some constants $C > 0$ and $r > 0$.

We will employ the quasi\hyp{}Monte Carlo quadrature
based on the Halton sequence, i.e.\ $\omega_i^{(N)} = 1/N$ and
$\xibfm_i^{(N)} = 2 \chibfm_i - \onebfm$, where $\chibfm_i$ denotes 
the $i$-th $M$-dimensional Halton point, cf.~\cite{H60}. Then, we know 
that, given that there exists an $\eps >0$ such that $\gamma_{\kappa, k} 
\leq c k^{-3-\eps}$ holds for some $c > 0$, for every $\delta > 0$ there 
exists a constant $C = C(\delta) > 0$
such that \eqref{eq:quadr} holds for $r = 1 - \delta$,
see e.g.\ \cite{HPS16} which is a consequence of \cite{W02}.
Clearly, other, possibly more sophisticated,
quadrature methods may also be considered,
for example, other quasi\hyp{}Monte Carlo quadratures,
such as those based on the Sobol sequence or other low\hyp{}discrepancy sequences
as well as their higher\hyp{}order adaptations,
and anisotropic sparse grid quadratures, see e.g.~\cite{DKLNS14,HHPS15,Niederreiter,S67}.

To approximate the mean of $u$ as in \eqref{eq:quad}, 
we require the values $u[\ybfm]$ for $\ybfm = \xibfm_i$. 
These values can be approximated by $u_{l}[\ybfm]$,
where $u_l$ is the Galerkin approximation of the spatially weak 
formulation on a finite dimensional subspace $V_l$ of $H_0^1(D)$;
that is, $u_l$ is the solution of
\begin{equation*}
  \left\{
  \begin{aligned}
    & \text{Find $u_l \in L_{\Pbbb_\ybfm}^\infty(\square; V_l)$ such that} \\
    & \quad \groupp[\big]{\Abfm[\ybfm] \Grad_\xbfm u_l[\ybfm], \Grad_\xbfm v}_{L^2(D; \Rbbb^d)} = \groupp[\big]{f, v}_{L^2(D; \Rbbb^d)} \\
    & \quad \quad \text{for almost every $\ybfm \in \square$ and all $v \in V_l$.}
  \end{aligned}
  \right.
\end{equation*}
We assume that a sequence of $V_l$ can be chosen for $l \in \Nbbb$
such that there is a constant $K$ with
\begin{equation}\label{eq:femr}
  \norm[\big]{u - u_l}_{L_{\Pbbb_\ybfm}^\infty(\square; H^1(D))} \leq K 2^{-\kappa l} .
\end{equation}
For example, we can consider $V_l$ to be the spaces of 
continuous finite elements of order $\kappa$ coming from
a sequence of quasi\hyp{}uniform meshes $\Tcal_l$ using 
isoparametric elements, where the mesh size behaves like $2^{-l}$.
Then, it is known from finite element theory that we have
\eqref{eq:femr} with $K \sim \norm{u}_{L_{\Pbbb_\ybfm}^\infty(\square; H^{\kappa+1}(D))}$,
see e.g.\ \cite{Braess,BrennerScott}.

The combination of the error estimates \eqref{eq:quadr} and
\eqref{eq:femr} then leads to
\begin{align*}
  \norm[\big]{\Mean[u] - Q_N[u_l]}_{H^1(D)}
  &\leq \norm[\big]{\Mean[u] - Q_N[u]}_{H^1(D)}
  + \norm[\big]{Q_N[u] - Q_N[u_l]}_{H^1(D)} \\
  &\leq \norm[\big]{\Mean[u] - Q_N[u]}_{H^1(D)}
  + \norm[\big]{Q_N[u - u_l]}_{H^1(D)} \\
  &\leq C N^{-r} + K 2^{-\kappa l} .
\end{align*}
Thus, choosing $N_l \isdef \groupc[\big]{2^{\kappa l/r}}$ finally yields
\begin{equation*}
  \norm[\big]{\Mean[u] - Q_{N_l}[u_l]}_{H^1(D)}
  \leq (C + K) 2^{-\kappa l} .
\end{equation*}
In contrast, the mixed regularity, shown before in Theorem
\ref{theorem:euybounds}, allows us to consider a multilevel
adaptation, which may be given as
\begin{equation*}
  \Mean[u]
  \approx Q_{l}^{\mathrm{ML}}[u_0, \ldots, u_l]
  \isdef \sum_{k=0}^{l} \Delta Q_{k}[u_{l-k}]
\end{equation*}
where
\begin{equation*}
  \Delta Q_{0} \isdef Q_{N_0}
  \quad\text{and}\quad
  \Delta Q_{k} \isdef Q_{N_k} - Q_{N_{k-1}}.
\end{equation*}
Indeed, this is the sparse grid combination technique as
introduced in \cite{GSZ92}, see also \cite{GHP15,HPS13}.
It thus follows that
\begin{equation*}
  \norm[\big]{\Mean[u] - Q_{l}^{\mathrm{ML}}[u_0, \ldots, u_l]}_{H^1(D)}
  \lesssim l 2^{-\kappa l}.
\end{equation*}

For complexity considerations, we shall consider a quadrature that is nested,
i.e.\ we may set $\xibfm_i = \xibfm_i^{(N)}$ as it does not depend on $N$.
Then, we note that $Q_{l}^{\mathrm{ML}}[u_0, \ldots, u_l]$ may explicitly 
be stated as
\begin{align*}
  Q_{l}^{\mathrm{ML}}[u_0, \ldots, u_l](\xbfm)
  &= \sum_{i=1}^{N_0} \omega_i^{(N_0)} u_l(\xbfm, \xibfm_i) \\
  &\quad {} + \sum_{k=1}^{l} \groupp[\Bigg]{
  \sum_{i=1}^{N_{k-1}} \groupp[\Big]{\omega_i^{(N_k)} - \omega_i^{(N_{k-1})}} u_{l-k}(\xbfm, \xibfm_i) \\
  &\quad \qquad\quad {} + \sum_{i=N_{k-1}+1}^{N_k} \omega_i^{(N_k)} u_{l-k}(\xbfm, \xibfm_i)} .
\end{align*}
Computing $Q_{N_l}[u_l]$ requires thus the values $u_{l, i}(\xbfm) 
\isdef u(\xbfm, \xibfm_i)$, which can be derived by solving
\begin{equation*}
  \left\{
  \begin{aligned}
    & \text{Find $u_{l, i} \in V_l$ such that} \\
    & \quad \groupp[\big]{\Abfm(\xibfm_i) \Grad_\xbfm u_{l, i}, \Grad_\xbfm v}_{L^2(D; \Rbbb^d)} = \groupp[\big]{f, v}_{L^2(D; \Rbbb^d)}\quad \text{for all $v \in V_l$.}
  \end{aligned}
  \right.
\end{equation*}
Generally, when considering a sequence of finite element 
spaces $V_l$ as described above, the number of degrees of 
freedom behaves like $\Ocal\groupp[\big]{2^{ld}}$ and computing 
one $u_{l, i}$ using state of the art methods will have a complexity 
that is $\Ocal\groupp[\big]{2^{ld}}$. As this has to be done $N_l$ 
times for the calculation of $Q_{N_l}[u_l]$, a complexity scaling 
is obtained that is $\Ocal\groupp[\big]{2^{l (\kappa/r+d)}}$.
Therefore, for the computation of the multilevel quadrature 
$Q_{l}^{\mathrm{ML}}[u_0, \ldots, u_l]$, we arrive at an over-all 
complexity of
\begin{equation*}
  \sum_{k=0}^{l} \sum_{i=1}^{N_{k}} \Ocal\groupp[\big]{2^{(l-k) d}}
  = \sum_{k=0}^{l} \Ocal\groupp[\big]{2^{k \kappa/r} 2^{(l-k) d}}
  = \begin{cases}
    \Ocal\groupp[\big]{2^{l \kappa/r}} & \text{for } d < \kappa/r , \\
    \Ocal\groupp[\big]{l 2^{l \max\{\kappa/r, d\}}} & \text{for } d =\kappa/r , \\
    \Ocal\groupp[\big]{2^{l d}} & \text{for } d > \kappa/r .
  \end{cases}
\end{equation*}
We mention that also non-nested quadrature formulae can be used
but lead to a somewhat larger constant in the complexity estimate,
see \cite{GHP15} for the details.

\begin{remark}
  If we redefine the $N_l$ as $N_l \isdef \groupc[\big]{l^{(1+\eps)/r} 2^{\kappa l/r}}$
  for an $\eps > 0$, then we have
  \begin{equation*}
    \norm[\big]{\Mean[u] - Q_{N_l}[u]}_{H^1(D)}
    \leq C \groupp[\big]{l^{(1+\eps)/r} 2^{\kappa l/r}}^{-r}
    = C l^{-(1+\eps)} 2^{-\kappa l}
  \end{equation*}
  and, as proposed in \cite{BSZ11}, we arrive at
  \begin{equation*}
    \norm[\big]{\Mean[u] - Q_{l}^{\mathrm{ML}}[u_0, \ldots, u_l]}_{H^1(D)}
    \lesssim 2^{-\kappa l} .
  \end{equation*}
  So, the logarithmic factor, which shows up in the convergence rate,
  can be removed by increasing the quadrature accuracy slightly faster.
  Note that this modification increases the hidden constant
  with a dependance on $\eps$.
\end{remark}

\begin{remark}
  In the particular situation of a standard quasi-Monte Carlo method,
  we can consider $\delta'$ such that $\delta > \delta' > 0$.
  Then, the quadrature error satisfies the estimate
  \begin{equation*}
    \norm[\big]{\Mean[u] - Q_{N_l}[u]}_{H^1(D)}
    \leq C_{\delta'} N_l^{\delta'-1}
    = C_{\delta'} 2^{-\kappa l} 2^{-\kappa l (\delta-\delta')/(1-\delta)}.
  \end{equation*}
  With a similar argument as in \cite{BSZ11},
  it follows that
  \begin{equation*}
    \norm[\big]{\Mean[u] - Q_{l}^{\mathrm{ML}}[u_0, \ldots, u_l]}_{H^1(D)}
    \lesssim 2^{-\kappa l} .
  \end{equation*}
  That is, the logarithmic factor, which shows up in the convergence rate,
  is removed at the cost of replacing the constant $C_{\delta}$ with $C_{\delta'}$
  and adds a constant with a dependance on $\delta'$
  yielding an increased hidden constant.
\end{remark}

While we have exclusively considered the case of the mean 
of the solution $u$ here, we do note that analogous statements 
may also be shown for example for the higher-order moments,
see \cite{HPS13} for instance.
%
%
\section{Numerical Results}
We will now consider two examples of the model problem 
\eqref{eq:sodp} with a diffusion coefficient of form 
\eqref{eq:AV} using the unit cube $D \isdef (0, 1)^3$ as 
the domain of computations. Therefore, in view of $H^2$-regularity
of the spatial problem under consideration, we are only considering 
the situation with $\kappa = 1$. In both examples, we set 
the global strength $a$ to $a \isdef 0.12$ and choose the 
right hand side $f \equiv 1$. For convenience, we define
\begin{equation*}
  s_j(\xbfm, \xbfm')\isdef 16\cdot x_j(1-x_j)\cdot x_j'(1-x_j').
\end{equation*}

\begin{example}
  In this first example, we choose 
  the description of $\Vbfm$ to be defined by
  \begin{equation*}
    \Mean[\Vbfm](\xbfm)
    \isdef \begin{bmatrix} 1 & 0 & 0 \end{bmatrix}^\trans
  \end{equation*}
  and
  \begin{equation*}
    \Cov[\Vbfm](\xbfm, \xbfm')
    \isdef \frac{1}{100}\exp\groupp[\Bigg]{-\frac{\norm[\big]{\xbfm-\xbfm'}_2^2}{50}} 
    \begin{bmatrix} 1& 0 & 0 \\ 0 & 9s_2(\xbfm, \xbfm')  & 0 \\ 0 & 0 & 9s_3(\xbfm, \xbfm')  \end{bmatrix} .
  \end{equation*}
  We note that for $j \in \{2, 3\}$ the covariance in the normal direction 
  on the parts of the boundary with $x_j \in \{0, 1\}$ is suppressed.
\end{example}

\begin{example}
  For this second example we choose 
  the description of $\Vbfm$ to be defined by
  \begin{equation*}
    \Mean[\Vbfm](\xbfm)
    \isdef \begin{bmatrix} \cos\big((x_3-0.5)\tfrac{\pi}{3}\big) \\ 
    	\sin\big((x_3-0.5)\tfrac{\pi}{3}\big) \\ 0 \end{bmatrix}
  \end{equation*}
  and
  \begin{equation*}
    \Cov[\Vbfm](\xbfm, \xbfm')
    \isdef \frac{9}{100}\exp\groupp[\Bigg]{-\frac{\norm[\big]{\xbfm-\xbfm'}_2^2}{50}} 
    \begin{bmatrix} s_1(\xbfm, \xbfm') & 0 & 0 \\ 0 & s_2(\xbfm, \xbfm') & 0 \\ 0 & 0 & s_3(\xbfm, \xbfm') \end{bmatrix} .
  \end{equation*}
  Here, the covariance in the normal direction 
  on all of the boundary is suppressed.
\end{example}

The numerical implementation is performed with aid of the 
problem\hyp{}solving environment DOLFIN \cite{FEniCSbook},
which is a part of the FEniCS Project \cite{FEniCSbook}.
The Karhunen-Lo\`eve expansion of the vector field $\Vbfm$ is computed
by the pivoted Cholesky decomposition, see \cite{HPS12,HPS15} for the details.
For the finite element discretisation,
we employ the sequence of nested triangulations $\Tcal_l$,
yielded by successive uniform refinement, i.e.\ cutting each tetrahedron into $8$ tetrahedra.
The base triangulation $\Tcal_0$ consists of $6 \cdot 2^3 = 48$ tetrahedra.
Then, we use interpolation
with continuous element\hyp{}wise linear functions
and the truncated pivoted Cholesky decomposition for the Karhunen\hyp{}Lo\`eve expansion 
approximation and continuous element\hyp{}wise linear functions in space.
The truncation criterion for the pivoted Cholesky decomposition is that the 
relative trace error is smaller than $10^{-4}\cdot 4^{-l}$.

Since the exact solutions of the examples are unknown,
the errors will have to be estimated. Therefore, in this section, 
we will estimate the errors for the levels $0$ to $5$ by substituting 
the exact solution with the approximate solution computed on the level 
$6$ triangulation $\Tcal_6$ using the quasi\hyp{}Monte Carlo quadrature 
based on Halton points with $10^4$ samples.

For every level, we also define 
the number of samples used by the quasi\hyp{}Monte Carlo method based on Halton points (QMC);
we choose
\begin{equation*}
  N_l \isdef \groupc[\Big]{2^{l/(1-\delta)} \cdot 10}
\end{equation*}
with $\delta \isdef 0.2$;
see Table~\ref{tab:nrealis3d} for the resulting values of $N_l$.
This then also implies the amount of samples used on the different
levels when using the multilevel quasi\hyp{}Monte Carlo method based on Halton points (MLQMC).
Based on these choices, we expect to see an asymptotic rate of 
convergence of $2^{-l}$ in the $H^1$\hyp{}norm for the mean 
and in the $W^{1,1}$\hyp{}norm for the variance.

\begin{table}[htb]
\centering
\begin{tabular}{lrrrrrr}
  \toprule
  $l$ & $0$ & $1$ & $2$ & $3$ & $4$ & $5$ \\
  \midrule
  $N_l$ & $10$ & $24$ & $57$ & $135$ & $320$ & $762$ \\
  \midrule
  $M_1$ & $17$ & $26$ & $30$ & $36$& $44$& $52$ \\
  $M_2$ & $14$ & $26$ & $30$ & $36$& $43$& $52$ \\
  \bottomrule
\end{tabular}
\caption{The number of samples for the first six levels and the respective parameter dimensions.}
\label{tab:nrealis3d}
\end{table}
Figures~\ref{fig:exp1err} and \ref{fig:exp2err} show the estimated errors of the solution's 
first moment on the left hand side and of the solution's second moment on the right hand 
side, each versus the discretisation level for the QMC and MLQMC quadrature
for the two different examples.
As expected, the QMC quadrature methods achieves the predicted rate of convergence
in both examples,
and this rate of convergence also carries over to its multilevel adaptation (MLQMC).

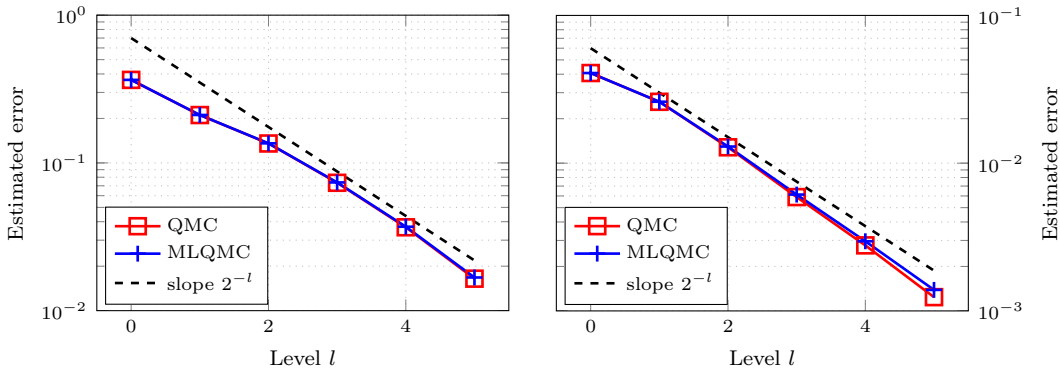
\begin{figure}[htb]
\centering
\begin{minipage}{0.49\textwidth}
\centering
\begin{tikzpicture}[baseline]
  \begin{semilogyaxis}[
      xlabel={Level $l$},
      ylabel={Estimated error},
      yticklabel pos=left,
      xmajorgrids=true,
      yminorgrids=true,
      ymax=1e0,
      ymin=1e-2,
    ]
    \addplot[red, mark=square, mark size=3pt, line width=1pt] table [x=L,y=QMC] {
L QMC MLQMC
0 0.364755 0.364755
1 0.210826 0.211084
2 0.135363 0.135666
3 0.0733695 0.0737691
4 0.0366668 0.0370021
5 0.0164656 0.0167718
    };
    \addlegendentry{QMC}
    \addplot[blue, mark=+, mark size=3pt, line width=1pt] table [x=L,y=MLQMC] {
L QMC MLQMC
0 0.364755 0.364755
1 0.210826 0.211084
2 0.135363 0.135666
3 0.0733695 0.0737691
4 0.0366668 0.0370021
5 0.0164656 0.0167718
    };
    \addlegendentry{MLQMC}
    \addplot[black, dashed, line width=1pt] table [x=L,y expr=0.7*2^(-\thisrow{L})] {
L QMC MLQMC
0 0.364755 0.364755
1 0.210826 0.211084
2 0.135363 0.135666
3 0.0733695 0.0737691
4 0.0366668 0.0370021
5 0.0164656 0.0167718
    };
    \addlegendentry{slope $2^{-l}$}
  \end{semilogyaxis}
\end{tikzpicture}
\end{minipage}
\begin{minipage}{0.49\textwidth}
\centering
\begin{tikzpicture}[baseline]
  \begin{semilogyaxis}[
      xlabel={Level $l$},
      ylabel={Estimated error},
      yticklabel pos=right,
      xmajorgrids=true,
      yminorgrids=true,
      ymax=1e-1,
      ymin=1e-3,
    ]
    \addplot[red, mark=square, mark size=3pt, line width=1pt] table [x=L,y=QMC] {
L QMC MLQMC
0 0.0407452  0.0407452 
1 0.0260228  0.0260457 
2 0.0128081  0.0129599 
3 0.00589495 0.00612986
4 0.00277783 0.00296096
5 0.00123954 0.00138782
    };
    \addlegendentry{QMC}
    \addplot[blue, mark=+, mark size=3pt, line width=1pt] table [x=L,y=MLQMC] {
L QMC MLQMC
0 0.0407452  0.0407452 
1 0.0260228  0.0260457 
2 0.0128081  0.0129599 
3 0.00589495 0.00612986
4 0.00277783 0.00296096
5 0.00123954 0.00138782
    };
    \addlegendentry{MLQMC}
    \addplot[black, dashed, line width=1pt] table [x=L,y expr=0.06*2^(-\thisrow{L})] {
L QMC MLQMC
0 0.0407452  0.0407452 
1 0.0260228  0.0260457 
2 0.0128081  0.0129599 
3 0.00589495 0.00612986
4 0.00277783 0.00296096
5 0.00123954 0.00138782
    };
    \addlegendentry{slope $2^{-l}$}
  \end{semilogyaxis}
\end{tikzpicture}
\end{minipage}
\caption{Example 1. $H^1$-error in the 1$^\text{st}$ moment (left) and $W^{1,1}$-error in the 2$^\text{nd}$ moment (right).}
\label{fig:exp1err}
\end{figure}
%
%
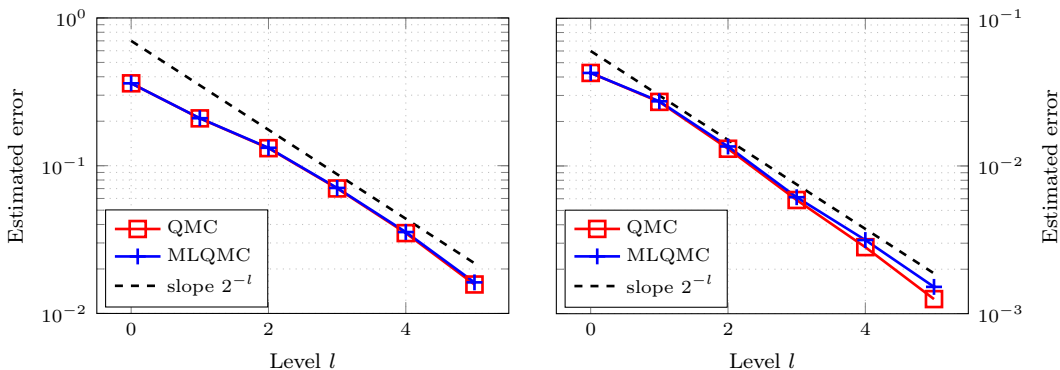
\begin{figure}[htb]
\centering
\begin{minipage}{0.49\textwidth}
\centering
\begin{tikzpicture}[baseline]
  \begin{semilogyaxis}[
      xlabel={Level $l$},
      ylabel={Estimated error},
      yticklabel pos=left,
      xmajorgrids=true,
      yminorgrids=true,
      ymax=1e0,
      ymin=1e-2,
    ]
    \addplot[red, mark=square, mark size=3pt, line width=1pt] table [x=L,y=QMC] {
L QMC MLQMC
0 0.361044  0.361044 
1 0.209126  0.209577 
2 0.131391  0.132163 
3 0.0701628 0.0706882
4 0.0349763 0.0355501
5 0.0157037 0.0162323
    };
    \addlegendentry{QMC}
    \addplot[blue, mark=+, mark size=3pt, line width=1pt] table [x=L,y=MLQMC] {
L QMC MLQMC
0 0.361044  0.361044 
1 0.209126  0.209577 
2 0.131391  0.132163 
3 0.0701628 0.0706882
4 0.0349763 0.0355501
5 0.0157037 0.0162323
    };
    \addlegendentry{MLQMC}
    \addplot[black, dashed, line width=1pt] table [x=L,y expr=0.7*2^(-\thisrow{L})] {
L QMC MLQMC
0 0.361044  0.361044 
1 0.209126  0.209577 
2 0.131391  0.132163 
3 0.0701628 0.0706882
4 0.0349763 0.0355501
5 0.0157037 0.0162323
    };
    \addlegendentry{slope $2^{-l}$}
  \end{semilogyaxis}
\end{tikzpicture}
\end{minipage}
\begin{minipage}{0.49\textwidth}
\centering
\begin{tikzpicture}[baseline]
  \begin{semilogyaxis}[
      xlabel={Level $l$},
      ylabel={Estimated error},
      yticklabel pos=right,
      xmajorgrids=true,
      yminorgrids=true,
      ymax=1e-1,
      ymin=1e-3,
    ]
    \addplot[red, mark=square, mark size=3pt, line width=1pt] table [x=L,y=QMC] {
L QMC MLQMC
0 0.0426113  0.0426113 
1 0.0271476  0.0273821 
2 0.0130775  0.0135457 
3 0.00588228 0.00614368
4 0.00282076 0.0031517 
5 0.00125573 0.00151914
    };
    \addlegendentry{QMC}
    \addplot[blue, mark=+, mark size=3pt, line width=1pt] table [x=L,y=MLQMC] {
L QMC MLQMC
0 0.0426113  0.0426113 
1 0.0271476  0.0273821 
2 0.0130775  0.0135457 
3 0.00588228 0.00614368
4 0.00282076 0.0031517 
5 0.00125573 0.00151914
    };
    \addlegendentry{MLQMC}
    \addplot[black, dashed, line width=1pt] table [x=L,y expr=0.06*2^(-\thisrow{L})] {
L QMC MLQMC
0 0.0426113  0.0426113 
1 0.0271476  0.0273821 
2 0.0130775  0.0135457 
3 0.00588228 0.00614368
4 0.00282076 0.0031517 
5 0.00125573 0.00151914
    };
    \addlegendentry{slope $2^{-l}$}
  \end{semilogyaxis}
\end{tikzpicture}
\end{minipage}
\caption{Example 2. $H^1$-error in the 1$^\text{st}$ moment (left) and $W^{1,1}$-error in the 2$^\text{nd}$ moment (right).}
\label{fig:exp2err}
\end{figure}
%
\section{Conclusion}
In this article,
we have considered the second order diffusion problem
\begin{equation*}
  \text{for almost every $\omega \in \Omega$: } \left\{
    \begin{alignedat}{2}
      - \Div_\xbfm \groupp[\big]{\Abfm[\omega] \Grad_\xbfm u[\omega]} &= f
      &\quad &\text{in }D , \\
      u[\omega] &= 0 &\quad &\text{on }\partial D ,
    \end{alignedat}
  \right.
\end{equation*}
with the uncertain diffusion coefficent given by
\begin{equation*}
  \Abfm[\omega] \isdef a \Ibfm + \groupp[\Big]{\norm[\big]{\Vbfm[\omega]}_2 - a} \frac{\Vbfm[\omega] \Vbfm^\trans[\omega]}{\Vbfm^\trans[\omega] \Vbfm[\omega]} .
\end{equation*}
This models anisotropic diffusion,
where the diffusion strength in the direction given by $\Vbfm / \norm{\Vbfm}_2$
is $\norm{\Vbfm}_2$ and perpendicular to it is $a$,
which can be used to model both diffusion in media that consist of thin fibres or thin sheets.

After having restated the problem in a parametric form
by considering the Karhunen\hyp{}Lo\`eve expansion of the random vector field $\Vbfm$,
we have shown that, given regularity of the elliptic diffusion problem,
the decay of the Karhunen\hyp{}Lo\`eve expansion of $\Vbfm$
entirely determines the regularity of the solution's dependence on the random 
parameter, also when considering this higher regularity in the spatial domain.

We then leverage this result to reduce the complexity of the approximation
of the solution's mean,
by using the multilevel quasi-Monte Carlo method
instead of the quasi-Monte Carlo method,
while still retaining the same error rate.
Indeed, while the QMC method yields a scheme,
where the uncertainty added increases the complexity,
this is not the case, when considering two or more spatial dimensions
and the MLQMC method.
That is, given elliptic regularity and up to a constant in the complexity,
adding uncertainty comes for free.
The numerical experiments corroborate these theoretical findings.

While we considered the use of QMC and its multilevel adaptation,
one can clearly also consider other quadrature methods, such as 
the anisotropic sparse grid quadrature, and then reduce the complexity 
by passing to their multilevel adaptations. Likewise, multilevel collocation
is also applicable.
%
%
\bibliographystyle{plain}
\bibliography{articles,books}
\end{document}